\documentclass[11pt,a4paper]{amsart}
\usepackage{url}
\usepackage{amssymb}
\usepackage[english]{babel}
\usepackage{enumerate}
\usepackage[T1]{fontenc}
\usepackage{graphicx}
\usepackage[utf8]{inputenc}
\usepackage{mathtools}
\usepackage{microtype}
\usepackage{multirow}
\usepackage{url}

\usepackage[hidelinks,hypertexnames=false]{hyperref}
\usepackage[nameinlink]{cleveref}

\DeclarePairedDelimiter{\abs}{\lvert}{\rvert}
\DeclarePairedDelimiterX{\ip}[2]{\langle}{\rangle}{#1,#2}
\DeclarePairedDelimiter{\norm}{\lVert}{\rVert}
\DeclarePairedDelimiter{\paren}{\lparen}{\rparen}	%This fixes some spacing issues that you get with just \left( \right). Plus, you can easily manually set the size.
\DeclarePairedDelimiter{\brac}{\lbrace}{\rbrace}
\DeclarePairedDelimiter{\brak}{\lbrack}{\rbrack}

\newcommand{\map}[3]{ #1 \colon #2 \to #3 }
\newcommand{\R}[0]{ \mathbb{R} }
\newcommand{\Q}[0]{ \mathbb{Q} }

\newcommand{\N}[0]{ \mathbb{N} }
\newcommand{\C}[0]{ \mathbb{C} }

\newcommand{\dx}{\, \mathrm{d}x}

\newcommand{\dz}{\, \mathrm{d}z}
\newcommand{\ds}{\, \mathrm{d}s}

\newcommand{\mc}[1]{\mathcal{#1}}
\newcommand{\mcF}{ \mathcal{F}}
\newcommand{\mcO}{ \mathcal{O}}
\newcommand{\mcU}{ \mathcal{U}}
\newcommand{\mcT}{ \mathcal{T}}
\newcommand{\mcL}{ \mathcal{L}}
\newcommand{\mcS}{ \mathcal{S}}

\DeclareMathOperator{\sgn}{sgn}
\DeclareMathOperator{\ran}{ran}
\DeclareMathOperator{\lspan}{span}

\newtheorem{theorem}{Theorem}[section]
\newtheorem{proposition}[theorem]{Proposition}
\newtheorem{lemma}[theorem]{Lemma}

\theoremstyle{definition}

\newtheorem{example}[theorem]{Example}

\theoremstyle{remark}
\newtheorem{remark}[theorem]{Remark}

\numberwithin{equation}{section}

\begin{document}
    
\title{Traveling gravity water waves with critical layers}

\author{Ailo Aasen}
\address{Department of Mathematical Sciences, Norwegian University of Science and Technology, 7491 Trondheim, Norway}
\email{ailo@stud.ntnu.no}

\author{Kristoffer Varholm}
\address{Department of Mathematical Sciences, Norwegian University of Science and Technology, 7491 Trondheim, Norway}
\email{kristoffer.varholm@math.ntnu.no}

\thanks{The first author was supported by a grant from the I. K. Lykke Fund. The authors also acknowledge the support of the project Nonlinear Water Waves by the Research Council of Norway (Grant No. 231668).}

\begin{abstract}
    We establish the existence of small-amplitude uni- and bimodal steady periodic gravity waves with an affine vorticity distribution, using a bifurcation argument that differs slightly from earlier theory. The solutions describe waves with critical layers and an arbitrary number of crests and troughs in each minimal period. An important part of the analysis is a fairly complete description of the local geometry of the so-called kernel equation, and of the small-amplitude solutions. Finally, we investigate the asymptotic behavior of the bifurcating solutions.
\end{abstract}

\maketitle

\section{Introduction}
    Up until fairly recently, most authors working with steady water waves have made the assumption that the vorticity
    \begin{equation}
        \label{vort}
        \omega \coloneqq v_x - u_y
    \end{equation}
    of the velocity field $(u,v)$ vanishes identically. Such waves are known as irrotational, as opposed to rotational waves where $\omega$ is allowed to be nonzero. Rotational waves can exhibit more exotic behavior than irrotational ones, including interior stagnation points and critical layers of closed streamlines~\cite{Ehrnstroem2012}. Stagnation points correspond to fluid particles that are stationary with respect to the wave, and for irrotational flows this can only occur at a sharp crest~\cite{Varvaruca2006}.

    Irrotational waves are mathematically simpler to work with than rotational ones, due to the existence of the velocity potential. The velocity potential is the harmonic conjugate of the stream function, thus enabling the use of tools such as complex analysis, which are typically not available with nonzero vorticity. The survey \cite{Toland1996} treats the theory of Stokes waves---an important class of irrotational waves---and the results on the so-called Stokes conjecture for such waves. This conjecture was not fully settled until the appearance of the paper \cite{Plotnikov2004}.
    
    Although rotational waves were considered intractable for mathematical analysis, they have long been important in more applied fields because rotational waves are not uncommon in nature: There are many physical effects that can induce rotation in waves, such as wind and thermal or salinity gradients \cite{Mei1984}, and rotational waves are also important in wave-current interactions \cite{Thomas1997}.

    The first, and still the only known, explicit example of a nontrivial traveling gravity water wave solution to the Euler equations was given in \cite{Gerstner1809} (see also \cite{Constantin2011c} for a more modern treatment) and is rotational; a fact which was only later pointed out by Stokes. Much later came the first existence result for small-amplitude waves with general vorticity distributions \cite{Dubreil-Jacotin1934}. It was not, however, before the pioneering article~\cite{Constantin2004} that large-amplitude waves were constructed, using an extension of the global bifurcation theory of Rabinowitz \cite{Rabinowitz1971,Healey1998}, leading to renewed interest in rotational waves. A corresponding result on deep water, where the lack of compactness is an obstacle, was established in \cite{Hur2006}.

    Due to the methods used, neither the waves in \cite{Dubreil-Jacotin1934} nor those in \cite{Constantin2004} exhibit stagnation. The first waves with a critical layer were constructed in \cite{Wahlen2009}, having constant vorticity. A different approach was used in \cite{Constantin2011b}, allowing for wave profiles with overhang (for which existence is still an open question, with some numerical evidence in the affirmative \cite{Vanden-Broeck1996}). The method of proof for the existence of nontrivial rotational waves is typically bifurcation from parallel flows with a prescribed vorticity distribution. Such parallel flows are described in great detail in~\cite{Kozlov2011}.

    Other authors have looked at waves with density stratification \cite{Escher2011,Henry2014,Walsh2014a}, waves with compactly supported vorticity \cite{Shatah2013,Varholm2016}, waves with discontinuous vorticity \cite{Constantin2011a}, and waves with a general vorticity distribution and stagnation \cite{Kozlov2014}. An upcoming result also establishes the existence of large-amplitude gravity water waves with a critical layer \cite{Constantin2014}. This was done in the presence of capillary effects in \cite{Matioc2014}, using an entirely different formulation.

    Of particular interest to us are \cite{Ehrnstroem2011,Ehrnstroem2015}, which cover small-amplitude waves with an affine vorticity distribution. This is the natural step up from the constant vorticity considered in \cite{Wahlen2009}, and the resulting waves can have an arbitrary number of critical layers \cite{Ehrnstroem2012}.

    In this paper, which builds upon \cite{Aasen2014}, we consider the same setting as in \cite{Ehrnstroem2011}. Small-amplitude solutions with an affine vorticity distribution are found by bifurcating from trivial solutions that depend naturally on three parameters. By using other choices for the bifurcation parameters in our argument, we obtain solution curves and sheets that, in general, do not coincide with those found in \cite{Ehrnstroem2011}. We are led to examine the asymptotic behavior of the bifurcating solutions; in particular for carefully chosen special cases. A complicating factor for our choice of bifurcation parameters is that they require an additional condition on the parameters. This condition can be interpreted as a nondegeneracy condition for the equation governing the dimension of the linearized problem.

    Another novel aspect of this work is a fairly complete description of the local geometry of the kernel of the operator appearing in the linearization. This is used to describe the geometry of the solution set near any trivial solution where the linear problem is one-dimensional, and for a class of trivial solutions with a two-dimensional linearized problem. We also show, by explicit construction, that the dimension of the linear problem can become arbitrary large for certain wavenumbers. This opens up the possibility for waves with arbitrarily many modes. Finally, we prove a regularity result, showing that the solutions we find are real analytic.

    The outline of the paper is as follows: In \Cref{sec:governing_equations} we formulate the problem and describe the setting in which we will work. Next, \Cref{sec:kernel} focuses on the kernel of the linearized operator. \Cref{1dBifSection} contains the bifurcation result for a one-dimensional kernel and gives the properties of the resulting bifurcation curves, while the final section, \Cref{2dBifSection}, covers two-dimensional bifurcation. Some useful derivatives are listed in \Cref{appendix}.

\section{The governing equations}
    \label{sec:governing_equations}
    We consider pure gravity waves. The fluid motion is assumed to be incompressible and two-dimensional, with the coordinate system oriented so that the $x$- and $y$-axes are horizontal and vertical, respectively. The fluid domain is bounded below by a flat bottom, and above by a free surface. Within this setting, our aim is to construct solutions of the steady water-wave problem; that is, to find a surface profile $\eta$ and a velocity field $(u,v)$, defined in the fluid domain
    \[
        \Omega_\eta \coloneqq \brac*{(x,y) \in \R^2:0<y<d+\eta(x)},
    \]
    where $d$ is the depth of the undisturbed fluid, satisfying the Euler equations
    \begin{subequations}
        \label{eq:euler}
        \begin{align}
            \label{masscons} u_x + v_y &= 0,  \\
            \label{Newt1} (u-c) u_x + v u_y   &= -p_x,  \\
            \label{Newt2} (u-c) v_x + v v_y     &= -p_y - g
        \end{align}
    \end{subequations}
    in $\Omega_\eta$. The surface profile is assumed to satisfy $\eta > -d$, so that the bottom is not exposed to air. In \eqref{Newt1} and \eqref{Newt2} the quantity $p$ is the pressure, $g$ is the gravitational acceleration, and $c$ is the constant velocity at which the wave travels.

    In addition to the equations in \eqref{eq:euler}, we impose the boundary conditions
    \begin{subequations}
        \label{eq:euler_bc}
        \begin{align}
            \label{tangFlowBC}  & & v &= 0 & &\text{at } y=0, & &  \\
            \label{KinBC}       & & v &= (u-c) \eta_x & & \multirow{2}[4]{*}[-2pt]{at $y = d+\eta(x)$.} & &  \\
            \label{PBC}         & & p &= 0 & &  & & 
        \end{align}
    \end{subequations}
    The first two boundary conditions are known as kinematic boundary conditions, and state that there is no flux through the surface or bottom. The dynamic boundary condition in \eqref{PBC} ensures that there is no jump in pressure across the free surface.
    
    We will be searching for periodic waves only, and so we introduce the wavenumber $\kappa>0$, and stipulate that all functions above be $2\pi/\kappa$-periodic in the horizontal variable.

    \subsection{Stream function formulation}
        We now reformulate the water wave problem \eqref{eq:euler}--\eqref{eq:euler_bc} in terms of a potential $\psi$, called the \textit{relative stream function}. From incompressibility \eqref{masscons}, together with $\Omega_\eta$ being simply connected, we know that there exists a function $\map{\psi}{\Omega_\eta}{\R}$ satisfying
        \[
            \psi_x = -v, \qquad \psi_y = u-c.
        \]
        This function is uniquely determined by $(u,v)$, up to a constant.
        
        The kinematic boundary condition \eqref{tangFlowBC} is equivalent to $\psi_x = 0$ at $y = 0$, and so $\psi$ is constant on the bottom. Similarly, we can use \eqref{KinBC} to deduce that $\psi$ is constant also on the surface. Next, \eqref{Newt1} and \eqref{Newt2} can be used to show that
        \[
            \{\psi,\Delta \psi\} = 0,
        \]
        where $\{\cdot,\cdot\}$ is the Poisson bracket defined by
        \[
            \{f,g\} \coloneqq f_y g_x - f_x g_y.
        \]
        Furthermore, by also using the boundary conditions in \eqref{eq:euler_bc}, one can infer that the surface Bernoulli equation
        \[
            \tfrac12 |\nabla \psi|^2 + g\eta = Q \qquad \text{on } y = d+\eta(x)
        \]
        holds for some $Q \in \R$.
        
        In terms of the stream function, the vorticity is given by
        \[
            \omega = -\Delta \psi,
        \]
        which follows directly from its definition in \eqref{vort}. Observe also that $\psi$ is $2\pi/\kappa$-periodic in the horizontal variable. To see this, note that $(x,y) \mapsto \psi(x + 2\pi/\kappa,y)$ is also a stream function, taking the same constant values as $\psi$ on the boundary. By uniqueness, they must be identical.
        
        The motivation for introducing the stream function is that, for a prescribed vorticity, the preceding equations are in fact equivalent to the steady water-wave problem. A precise statement, taken from \cite{Ehrnstroem2012}, can be found in \Cref{SFForm} below. We will use a subscript $\kappa$ to denote $2\pi/\kappa$-periodicity in the horizontal variable.
        
        \begin{proposition}[Stream function \cite{Ehrnstroem2012}]  \label{SFForm}
            For $\eta \in C_\kappa^3(\R)$, $u,v \in C_\kappa^2(\overline \Omega_\eta)$ and a prescribed vorticity $\omega \in C_\kappa^1(\overline{\Omega}_\eta)$, the steady water-wave problem \eqref{eq:euler}--\eqref{eq:euler_bc} is equivalent to the stream function formulation
            \begin{subequations}
                \label{eq:euler_stream}
                \begin{align}
                    \Delta \psi &= -\omega                  && \multirow{2}[4]{*}[-2pt]{in $\Omega_\eta$,} \notag\\
                    \{ \psi,\Delta \psi \} &= 0             && \label{eq:stream_poisson}\\
                    \psi &= m_0                             && \text{at $y=0$}, \notag\\
                    \psi &= m_1                             && \multirow{2}[4]{*}[-2pt]{at $y = d+\eta(x)$,} \notag\\
                    \tfrac12 |\nabla \psi|^2 + g\eta &= Q	&& \notag
                \end{align}
            \end{subequations}
            for $\psi \in C_\kappa^3(\overline \Omega_\eta)$ and constants $m_0, m_1$ and $Q$.
        \end{proposition}

    \subsection{The vorticity distribution}
        As long as the fluid velocity does not exceed the wave velocity, so there is no stagnation, the vorticity at a point only depends on the value of the stream function at that point. This dependency is described by what is known as the vorticity distribution.
        \begin{lemma}[Vorticity distribution \cite{Constantin2004}] \label{gammaExists}
            Suppose that $u<c$. Then there exists a function $\gamma$ such that $\omega = \gamma(\psi)$ in $\Omega_\eta$.
        \end{lemma}
        
        A notable consequence of the existence of a vorticity distribution is that \Cref{eq:stream_poisson} is trivially satisfied, because
        \[
            \{ \psi,\Delta \psi \} = \psi_y (-\gamma(\psi))_x - \psi_x (-\gamma(\psi))_y =0
        \]
        by the chain rule. Observe also that the condition in \Cref{gammaExists} is sufficient, but not necessary. By \emph{assuming} the existence of a vorticity distribution, we will still obtain solutions of the water-wave problem, even if $u < c$ is not satisfied. In fact, the solutions that we will find can exhibit stagnation and critical layers. The introduction of the vorticity distribution is standard for rotational waves, and was used already in \cite{Dubreil-Jacotin1934}.
        
        We shall consider the case where $\gamma$ is affine. By making a shift of $\psi$, it is sufficient to consider the case of linear $\gamma$. After scaling to unit depth and scaling away the gravitational acceleration, the stream function formulation \eqref{eq:euler_stream} reduces to
        \begin{subequations}
            \label{problem}
            \begin{align}
            \Delta \psi &= \alpha \psi				&& \text{in } \Omega_\eta,\label{ProbOmeEq} \\
            \tfrac12 |\nabla \psi|^2 + \eta &= Q	&& \text{on } S,\label{ProbSEq} \\
            \psi &= m_0								&& \text{on } B, \label{ProbBCond}\\
            \psi &= m_1								&& \text{on } S, \label{ProbSCond}
            \end{align}
        \end{subequations}
        where we have introduced the bottom $B \coloneqq \{(x,y): y=0 \}$ and the surface $S \coloneqq \{ (x,y): y = 1+\eta(x)\}$. The parameter $\alpha$ in \eqref{ProbOmeEq} controls the vorticity, and will be assumed to be \emph{negative}. For positive $\alpha$, one-dimensional---but not higher-dimensional---bifurcation is possible. More discussion on this can be found in \cite{Ehrnstroem2011}.
        
        Observe now that the system \eqref{problem} makes sense also in less regular function spaces than those specified in \Cref{SFForm}, and we will therefore allow for less regular (but still classical) solutions. More precisely, we will search for solutions 
        \[
            \eta \in C^{2,\beta}_{\kappa,\textnormal{e}}(\R) \qquad \text{and} \qquad \psi \in C^{2,\beta}_{\kappa,\textnormal{e}}(\overline \Omega_\eta),
        \]
        where $\beta \in (0,1)$, and the subscript e signifies the subspace of functions which are even in the horizontal variable. The motivation for working in these Hölder spaces is that \Cref{mcLisFred} then holds.
        
        \begin{remark}[Regularity]
            \label{rem:optimal_regularity}
            Due to \Cref{ProbOmeEq} and elliptic regularity for the differential operator $\alpha - \Delta$, the stream function $\psi$ is analytic in $\Omega_\eta$. In fact, we show in \Cref{thm:optimal_regularity} that this is true even up to the boundary.
        \end{remark}

    \subsection{Trivial solutions and flattening}
        The solutions of \eqref{problem} that we shall construct will be small perturbations of steady flows that are parallel to the bottom. These parallel flows are the trivial solutions of \eqref{problem}, in the sense that $\eta = 0$ and the stream function $\psi$ only depends on $y$. By integrating \Cref{ProbOmeEq}, we arrive at trivial solutions of the form
        \begin{equation}
            \label{Laminar}
            \psi_0(y, \Lambda) \coloneqq\mu \cos \paren[\big]{|\alpha|^{1/2} (y-1)+\lambda}, \qquad \Lambda = (\mu,\alpha,\lambda)\in \R^3,
        \end{equation}
        with corresponding $Q(\Lambda)$, $m_0(\Lambda)$ and $m_1(\Lambda)$ determined from \eqref{ProbSEq}--\eqref{ProbSCond} as
        \begin{equation}
            \label{eq:trivial_solution_constants}
            Q(\Lambda) = \frac{\mu^2 |\alpha| \sin^2(\lambda)}{2},\quad
            \begin{aligned}
                m_0(\Lambda) &= \mu \cos(\lambda-|\alpha|^{1/2}),\\
                m_1(\Lambda) &= \mu \cos(\lambda).
            \end{aligned}
        \end{equation}
        
        Our goal is to find nontrivial solutions of \eqref{problem} for certain values of $\Lambda$, corresponding to these particular values of $Q$, $m_0$ and $m_1$. For technical reasons which we will elucidate later in \Cref{rem:fredholm}, it is assumed that
        \begin{equation}
            \label{psiAssumption}
            \psi_{0y}(1) = -\mu |\alpha|^{1/2} \sin(\lambda) \neq 0.
        \end{equation}
        As in \eqref{psiAssumption}, we will often omit the dependence on $\Lambda$ in our notation.
        
        The main difficulty with the system \eqref{problem} is that it is a free-boundary problem, which entails that the domain is a priori unknown. There are several ways of fixing the domain. Here, we will use the ``naive'' flattening transform
        \[
            G:(x,y) \mapsto \paren*{ x, \frac{y}{1+\eta(x)} },
        \]
        giving a bijection from the sets $\Omega_\eta$, $B$ and $S$ onto
        \[
            \hat{\Omega} = \{(x,s): s \in [0,1] \}, \quad \hat B \coloneqq \{(x,s): s=0 \}\quad \text{and}\quad  \hat S \coloneqq \{(x,s):s=1\},
        \]
        respectively. Using that $\eta \in C^{2,\beta}_{\kappa,\textnormal{e}}(\R)$, we find that the map $G$ is a $C^{2,\beta}$-diffeomorphism, with inverse given by
        \[
            G^{-1}(x,s) = \paren[\big]{x,(1+\eta(x))s}.
        \]
        
        If we define $\hat{\psi}$ on $\hat{\Omega}$ by $\hat{\psi} \coloneqq \psi \circ G^{-1}$, then \eqref{ProbSEq} and \eqref{ProbOmeEq} become 
        \begin{equation}
            \label{mcEDef}
            \begin{aligned}
                \paren[\bigg]{\partial_x - \frac{s\eta_x}{1+\eta}\partial_s}^2\hat{\psi} + \frac{\hat{\psi}_{ss}}{(1+\eta)^2} &=\alpha\hat{\psi}  &\text{in $\hat{\Omega}$},\\
                \frac{(1+\eta_x^2)\hat{\psi}_s^2}{2(1+\eta)} + \eta &=Q  &\text{on $\hat{S}$},
            \end{aligned}
        \end{equation}
        in the new flattened variables, for which we have the following:
        \begin{lemma}[Equivalence \cite{Ehrnstroem2011}] \label{mcELem}
            For functions $\eta \in C^{2,\beta}_{\kappa,\textnormal{e}}(\R)$ and $\psi \in C_{\kappa,\textnormal{e}}^{2,\beta}(\overline{\Omega}_\eta)$, the stream function formulation \eqref{problem} is equivalent to the transformed problem in \eqref{mcEDef} for  $\eta \in C^{2,\beta}_{\kappa,\textnormal{e}}(\R)$ and
            \[
                \brac[\big]{\hat{\psi} \in C^{2,\beta}_{\kappa,\textnormal{e}}(\overline{\hat{\Omega}}):\hat{\psi}|_{s=0} = m_0,\, \hat{\psi}|_{s=1}=m_1}.
            \]
            Moreover, in this setting, a pair $(\eta,\hat{\psi}) = (0,\hat{\psi}(s))$ solves \eqref{mcEDef} if and only if $\hat \psi = \psi_0$.
        \end{lemma}
        
        With the trivial solutions found and the flattening transform introduced, we now elaborate on \Cref{rem:optimal_regularity}. Any solution which is sufficiently close to a trivial solution is in fact analytic, as long as \eqref{psiAssumption} holds. The precise statement can be found in \Cref{thm:optimal_regularity} below.
        
        \begin{theorem}[Regularity]
            \label{thm:optimal_regularity}
            Suppose that a solution $(\eta,\psi)$ of the problem \eqref{problem} in $C^1(\R) \times C^2(\R)$ is such that the normal derivative $\partial_n\psi$ of the stream function vanishes at no point on the surface. Then we have the following:
            \begin{enumerate}[(i)]
                \item The surface profile $\eta$ is analytic.
                \item The stream function $\psi$ extends to an analytic function on an open set containing $\overline{\Omega}_\eta$.
            \end{enumerate}
            The assumption on $\partial_n \psi$ holds when $\Lambda$ satisfies \eqref{psiAssumption} and $\hat{\psi}$ is sufficiently close to $\psi_0(\cdot,\Lambda)$ in $C^2(\overline{\hat{\Omega}})$.
        \end{theorem}
        \begin{proof}
            We start by showing that $\eta$ is analytic. For this, we will use the approach taken in \cite{Constantin2011}, which is to apply \cite[Theorem 3.2]{Kinderlehrer1978}. In \cite{Constantin2011} this was done under the assumption of no stagnation, but it is sufficient to assume that stagnation does not occur on the surface. This corresponds to the Shapiro--Lopatinski\u{\i} condition for a certain elliptic system.
            
            Let $\Omega_\eta^+$ be the component of $\R^2 \setminus S$ that does not contain $\Omega_\eta$. Proceed to define the function $\map{u}{\Omega_\eta \cup S \cup \Omega_\eta^+}{\R}$ by
            \[
                u(x,y) \coloneqq \begin{cases}
                    0 & (x,y) \in S \cup \Omega_\eta^+,\\
                    \psi(x,y) - m_1  & (x,y) \in \Omega_\eta \cup S,
                \end{cases}
            \]
            and the differential operator $L$ by $L\coloneqq\alpha - \Delta$. Observe that \Cref{ProbSEq,ProbSCond} imply that
            \[
                f(y,\partial_n \psi) \coloneqq \frac{1}{2} (\partial_n \psi)^2 + y - 1 - Q = 0
            \]
            on $S$. All the assumptions of \cite[Theorem 3.2]{Kinderlehrer1978} are now satisfied, with $G \coloneqq L$ and $F(u) \coloneqq Lu + \alpha m_1$ (see the remark immediately after the theorem). We conclude that $\eta$ is analytic.
            
            Note now that the differential operator $L$ is strongly elliptic in the sense of \cite[Equation (1.7)]{Morrey1957}. Equipped with the fact that $\eta$ is analytic, we can use \cite[Theorem A]{Morrey1957} to conclude that $\psi$ extends to an analytic function on an open set containing $\overline{\Omega}_\eta$.
            
            The final part of the theorem follows because
            \[
                \partial_n \psi = \sqrt{1+(\eta')^2}\psi_y(\cdot,\eta)=\frac{\sqrt{1+(\eta')^2}}{1+\eta} \hat{\psi}_s(\cdot,1),
            \]
            where $\hat{\psi}_s(\cdot,1)$ is bounded away from $0$ as long as $\hat{\psi}$ is sufficiently close to $\psi_0$ in $C^2(\overline{\hat{\Omega}})$, due to the assumption that \Cref{psiAssumption} holds.
        \end{proof}
        \begin{remark}
            \Cref{thm:optimal_regularity} is a local result at heart. It is clear from the proof that if $\partial_n \psi(x_0,\eta(x_0))\neq 0$, then $\eta$ is analytic in a neighborhood of $x_0$. This, in turn, implies that $\psi$ extends analytically across the surface near the point $(x_0,\eta(x_0))$.
        \end{remark}
        \begin{remark}
            Recall that the stream function $\psi$ is analytic on $\Omega_\eta$, regardless of whether the condition on $\partial_n \psi$ on the surface in \Cref{thm:optimal_regularity} is satisfied. It is worth noting that this implies, through the implicit function theorem, that the streamlines are analytic curves away from stagnation points.
        \end{remark}

    \subsection{The linearized problem} \label{funcAnSet}
        In order to linearize \Cref{mcEDef} around a trivial solution $\psi_0$, we write $\hat{\psi}=\psi_0+\hat{\phi}$, and introduce the spaces
        \[
            X = X_1 \times X_2 \coloneqq C^{2,\beta}_{\kappa,\textnormal{e}}(\R) \times \brac[\big]{\hat{\phi} \in C^{2,\beta}_{\kappa,\textnormal{e}}(\overline{\hat{\Omega}}):\hat{\phi}|_{s=0} = \hat{\phi}|_{s=1}=0 }
        \]
        and
        \[
            Y=Y_1 \times Y_2 \coloneqq C^{1,\beta}_{\kappa,\textnormal{e}}(\R) \times C^{\beta}_{\kappa,\textnormal{e}}(\overline{\hat{\Omega}}).
        \]
        We will write $w=(\eta,\hat{\phi})$ for elements of $X$. To capture our assumptions, it is convenient to define the sets
        \[
            \mcO\coloneqq \{ w \in X:\min \eta >-1 \},
        \]
        and, to enforce that $\alpha<0$ and \eqref{psiAssumption} hold,
        \[
            \mcU \coloneqq \brac*{ (\mu,\alpha,\lambda) \in \R^3:\mu \neq 0,\, \alpha < 0,\, 0 < \lambda < \pi}.
        \]
        
        We now define the map $\map{\mcF = (\mcF_1,\mcF_2)}{\mcO \times \mcU}{Y}$ by
        \begin{subequations}
            \label{mcFDef}
            \begin{align}
                \mcF_1(w,\Lambda) &\coloneqq \frac{(1+\eta_x^2)(\psi_{0s} + \hat{\phi}_s)^2}{2(1+\eta)^2} + \eta - Q(\Lambda), \label{mcF1Def}\\
                \mcF_2(w,\Lambda) &\coloneqq \paren[\bigg]{\partial_x - \frac{s\eta_x}{1+\eta}\partial_s}^2(\psi_0 + \hat{\phi}) +\frac{\psi_{0ss}+\hat{\phi}_{ss}}{(1+\eta)^2} - \alpha(\psi_{0}+\hat{\phi}), \notag
            \end{align}
        \end{subequations}
        where $\psi_0$ is as in \Cref{Laminar} and $Q(\Lambda)$ is given in \eqref{eq:trivial_solution_constants}. In \eqref{mcF1Def}, it is understood that the functions $\psi_{0s}$ and $\hat{\phi}_s$ are evaluated at $s=1$. It is clear that $\mcF$ is well defined and smooth as a map $\mcO \times \mcU \to Y$. We wish to solve the equation
        \begin{equation}
            \label{eq:F_problem}
            \mcF(w,\Lambda) = 0.
        \end{equation}
        
        We obtain the linearized problem by taking the partial derivative of $\mcF$ with respect to $w$ at the point $(0,\Lambda)$. This yields
        \begin{subequations}
            \begin{align}
                D_w \mcF_1 (0,\Lambda) w &= \psi_{0s} \hat{\phi}_s-\psi_{0s}^2 \eta + \eta, \label{DF1w}\\
                D_w \mcF_2 (0,\Lambda) w &= (\Delta-\alpha) \hat{\phi}-s \psi_{0s} \eta_{xx} -2 \psi_{0ss} \eta \notag,
            \end{align}
        \end{subequations}
        where it again is understood that the functions are evaluated at $s=1$ in \eqref{DF1w}. By introducing an isomorphism, in \Cref{mcTProp} below, we can transform $D_w \mcF$ into a simpler elliptic operator. For this purpose, define
        \[
            \tilde{X}_2 \coloneqq \brac[\big]{ \phi \in C^{2,\beta}_{\kappa,\textnormal{e}}(\overline{\hat{\Omega}}): \phi|_{s=0} = 0 }, \quad \tilde{X} \coloneqq X_1 \times \tilde{X}_2,
        \]
        where we have the inclusion $X \subset \tilde X \subset Y$. We will typically use the letter $\phi$ for elements of $\tilde X_2$.
        \begin{proposition}[The $\mathcal{T}$ isomorphism \cite{Ehrnstroem2011}] \label{mcTProp}
            The bounded linear operator $\mathcal{T}(\Lambda):\tilde{X}_2 \to X$ defined by
            \[
                \mathcal{T}(\Lambda) \phi =(\eta_{\phi},\hat{\phi}) \coloneqq \paren*{ -\frac{\phi|_{s=1}}{\psi_{0s}(1)}, \phi-\frac{s \psi_{0s}}{\psi_{0s}(1)} \phi|_{s=1} }
            \]
            is an isomorphism of Banach spaces, and the operator
            \[
                \mcL (\Lambda) = (\mcL_1(\Lambda),\mcL_2(\Lambda)) \coloneqq D_w \mcF(0,\Lambda)\mcT(\Lambda):\tilde{X}_2 \to Y
            \]
            satisfies
            \begin{equation}
                \label{mcLExp}
                \mcL(\Lambda) \phi = \paren*{\brak*{ \psi_{0s} \phi_s- \paren*{\psi_{0ss}+\frac{1}{\psi_{0s}}} \phi }_{s=1}, (\Delta - \alpha) \phi }.
            \end{equation}
        \end{proposition}
        
        \begin{proof}
            That $\mcT$ is well defined and an isomorphism is almost immediate. The expression for $\mcL(\Lambda)$ in \eqref{mcLExp} follows by direct computation.
        \end{proof}

\section{The kernel and dimensional reduction}
    \label{sec:kernel}
    
    Introduce the complex parameter
    \[
        \theta_n = \theta(n,\alpha) \coloneqq \sqrt{\alpha+ n^2\kappa^2} = \begin{cases} \sqrt{n^2\kappa^2-\abs{\alpha}},  & n \geq \abs{\alpha}^{1/2}/\kappa, \\
            i\sqrt{\abs{\alpha}-n^2 \kappa^2}  & n < \abs{\alpha}^{1/2}/\kappa,
        \end{cases}
    \]
    for nonnegative integers $n$. This parameter will appear in functions of the form $\cosh(\theta_n s)$ and $\sinh(\theta_n s)/\theta_n$, which are always real-valued. We record that
    \begin{align*}
        \cosh(\theta_n s) &= \begin{cases} \mathrlap{\cosh(\abs{\theta_n} s),}\hphantom{\sinh(\abs{\theta_n} s)/\abs{\theta_n},} & n \geq \abs{\alpha}^{1/2}/\kappa, \\
            \cos(|\theta_n|s),   & n < \abs{\alpha}^{1/2}/\kappa,
        \end{cases}\\
        \frac{\sinh(\theta_n s)}{\theta_n} &= \begin{cases} \sinh(\abs{\theta_n} s)/\abs{\theta_n}, & n \geq \abs{\alpha}^{1/2}/\kappa, \\
            \sin(|\theta_n| s)/|\theta_n|,   & n < \abs{\alpha}^{1/2}/\kappa.
        \end{cases}
    \end{align*}
    In the event that $\theta_n = 0$, we will interpret expressions with $\theta_n$ as extended by continuity. In particular, $\sinh(\theta_n s)/\theta_n$ is interpreted as $s$. 
    
    We now describe the kernel of $\mcL(\Lambda)$, which is directly related to the kernel of $D_w\mcF(0,\lambda)$ through $\mcT(\Lambda)$, in terms of the above functions. The following proposition is stated, but not proved, in \cite{Ehrnstroem2011}. We include its proof because it is instructive.
    \begin{proposition}[Kernel of $\mcL(\Lambda)$ \cite{Ehrnstroem2011}] \label{kernelLemma}
        Let $\Lambda \in \mcU$. A basis for $\ker \mcL(\Lambda)$ is then given by $\{ \phi_n \}_{n \in M}$, where
        \begin{equation}
            \label{phik}
            \phi_n(x,s) \coloneqq \cos(n\kappa x) \frac{\sinh(\theta_n s)}{\theta_n}
        \end{equation}
        and $M=M(\Lambda)$ is the finite set of all $n \in \N_0$ satisfying the \emph{kernel equation}
        \begin{equation}
            \label{KerEq}
            l(n,\alpha) = r(\Lambda),
        \end{equation}
        where
        \begin{align}
            l(n,\alpha) &\coloneqq \theta_n \coth (\theta_n), \notag \\
            r(\Lambda) &\coloneqq \frac{1}{\mu^2 |\alpha| \sin^2(\lambda)}+ |\alpha|^{1/2} \cot(\lambda). \label{eq:kernel_equation_rhs}
        \end{align}
    \end{proposition}
    \begin{proof}
        Suppose that $\phi \in \ker \mcL(\Lambda)$, and expand it in a Fourier series
        \[
            \phi(x,s) = \sum_{n=0}^{\infty} a_n(s) \cos(n\kappa x).
        \]
        From $\mcL_2(\Lambda)\phi=0$, we deduce that the coefficients satisfy
        \begin{equation}
            \label{anDE}
            a_n''(s)-\theta_n^2 a_n(s)= 0, \quad s \in (0,1),
        \end{equation}
        while $\phi|_{s=0}=0$ and $\mcL_1(\Lambda)\phi=0$ yield the boundary conditions
        \begin{subequations}
            \begin{gather}
                \label{eq:an_dirichlet}
                a_n(0) =0,\\
                \label{eq:an_robin}
                \psi_{0s}(1)a_n'(1)-\paren*{\psi_{0ss}(1) + \frac{1}{\psi_{0s}(1)}}a_n(1)=0,
            \end{gather}
        \end{subequations}
        for all $n \geq 0$.
        
        The general solution of \eqref{anDE} with the boundary condition \eqref{eq:an_dirichlet} is
        \[
            a_n(s) = B_n \frac{\sinh(\theta_n s)}{\theta_n}, \quad B_n \in \R,\, n \geq 0,
        \]
        for which the Robin condition \eqref{eq:an_robin} reduces to
        \[
            \paren*{ \psi_{0s}(1)\cosh(\theta_n) - \paren*{ \psi_{0ss}(1) + \frac{1}{\psi_{0s}(1)} } \frac{\sinh(\theta_n)}{\theta_n} } B_n = 0.
        \]
        Hence, if $B_n$ (and thus $a_n$) is nonzero, then
        \begin{equation}
            \label{KEcatter}
            \psi_{0s}(1)\cosh(\theta_n) - \paren*{ \psi_{0ss}(1) + \frac{1}{\psi_{0s}(1)} } \frac{\sinh(\theta_n)}{\theta_n} = 0
        \end{equation}
        must hold. Observe that \Cref{KEcatter} implies that $\sinh(\theta_n)/\theta_n \neq 0$; otherwise we would have $\cosh(\theta_n) =\sinh(\theta_n)=0$, and therefore $\exp(\theta_n)=0$. Thus, by inserting the definition \eqref{Laminar} of $\psi_0$ into \Cref{KEcatter}, we arrive at \eqref{KerEq}. This condition is also sufficient for $\phi_n$ to lie in the kernel.
        
        The set $M$ of $n \in \N_0$ such that \eqref{KerEq} holds is finite, because the function $l(\cdot,\alpha)$ is strictly increasing as soon as $n \geq \abs{\alpha}^{1/2}/\kappa$.
    \end{proof}
    \begin{remark}
        \label{remark:nontrivial_solutions}
        In order to get nontrivial solutions, $\Lambda$ should be chosen such that $0 \notin M(\Lambda)$. The function $\phi_0$, see \eqref{phik}, does not depend on $x$.
    \end{remark}
    
    The next lemma, inspired by \cite[Theorem IV.5.17]{Kato1995}, serves to show that the set-valued map $\map{M}{\mcU}{2^{\N_0}}$ defined in \Cref{kernelLemma} is upper semicontinuous. This implies that no new solutions of the kernel equation \eqref{KerEq} can appear if $\Lambda$ is perturbed slightly.
    
    \begin{lemma}[Upper semicontinuity]
        \label{lemma:M_decreasing}
        
        Let $\Lambda^* \in \mcU$. Then
        \[
            M(\Lambda) \subset M(\Lambda^*)
        \]
        for all $\Lambda$ in a neighborhood of $\Lambda^*$.
    \end{lemma}
    \begin{proof}
        Suppose that this is not the case. Then we can construct a sequence $(\Lambda_i)_{i \in \N}$ converging to $\Lambda^*$, and a corresponding sequence $(n_i)_{i \in \N}$ such that $n_i \notin M(\Lambda^*)$ and $l(n_i,\alpha_i)=r(\Lambda_i)$ for all $i \in \N$. By the continuity of $r$ at $\Lambda^*$, the sequence $(r(\Lambda_i))_{i\in \N}$, and therefore $(l(n_i,\alpha_i))_{i \in \N}$, is bounded. This implies that $(n_i)_{i \in \N}$ is bounded, so we may assume that it is constant. Thus there is an $n \notin M(\Lambda^*)$ such that $l(n,\alpha_i) = r(\Lambda_i)$ for all $i \in \N$. The boundedness of the sequences now ensures that $l(n,\cdot)$ is well-defined and continuous at $\alpha^*$. We conclude that $l(n,\alpha^*) = r(\Lambda^*)$, which contradicts $n \notin M(\Lambda^*)$.
    \end{proof}
    
    \begin{figure}[htb]
        \centering
        \includegraphics{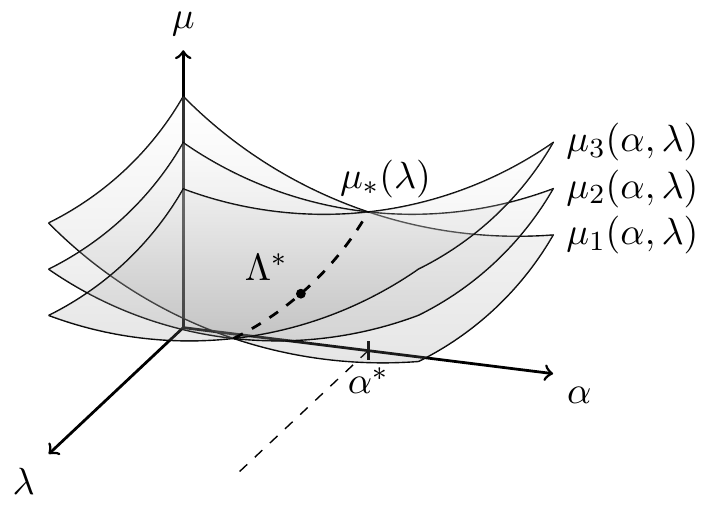}
        \caption{An illustration of \Cref{theorem:kernel_local_description} when $|M(\Lambda^*)|=3$. A full description of the kernel can be given near $\Lambda^*$.}
        \label{fig:kernel_geometry}
    \end{figure}
    
    We can now use \Cref{lemma:M_decreasing} to give a local description of the structure of the kernel equation. This will be useful when describing the solution set of \eqref{eq:F_problem}. See also \Cref{fig:kernel_geometry}.
    
    \begin{theorem}[Local description]
        \label{theorem:kernel_local_description}
        Suppose that $M(\Lambda^*) = \{n_1,\ldots,n_N\}$. Then we may define
        \begin{equation}
            \label{eq:definition_mui}
            \mu_i(\alpha,\lambda) \coloneqq \frac{\sgn(\mu^*)}{\abs{\alpha}^{1/2}\sin(\lambda)(l(n_i,\alpha)-\abs{\alpha}^{1/2}\cot(\lambda))^{1/2}}, \quad 1 \leq i \leq N,
        \end{equation}
        on a neighborhood of $(\alpha^*,\lambda^*)$,
        \[
            \mu_*(\lambda) \coloneqq \mu_1(\alpha^*,\lambda) \quad (=\cdots=\mu_N(\alpha^*,\lambda))
        \]
        on a neighborhood of $\lambda^*$, and we have
        \[
        M(\Lambda) = \begin{cases}
        M(\Lambda^*) & \alpha = \alpha^*, \mu = \mu_*(\lambda),\\
        \{n_i\} & \alpha \neq \alpha^*, \mu = \mu_i(\alpha,\lambda),\\
        \varnothing & \text{otherwise,}
        \end{cases}
        \]
        for all $\Lambda$ in a neighborhood of $\Lambda^*$.
    \end{theorem}
    \begin{proof}
        By \Cref{lemma:M_decreasing} there is a neighborhood of $\Lambda^*$ in which $M(\Lambda)$ is the set of $n_i\in M(\Lambda^*)$ for which $l(n_i,\alpha) = r(\Lambda)$. Observe now that $\Lambda$ sufficiently close to $\Lambda^*$ we have $l(n_i,\alpha) = r(\Lambda)$ if and only if $\mu = \mu_i(\alpha,\lambda)$, where $\mu_i$ is as in \eqref{eq:definition_mui}. Moreover, if $i \neq j$ then $l(n_i,\cdot)-l(n_j,\cdot)$ is a nonzero analytic function on a neighborhood of $\alpha^*$. It follows that we may choose the neighborhood of $\Lambda^*$ in such a way that the only intersection of the graphs of the $\mu_i$ occurs when $\alpha =\alpha^*$.
    \end{proof}
    
    The bifurcation results in \Cref{1dBifSection,2dBifSection} are valid under the assumption that $\ker \mcL(\Lambda)$ is respectively one- and two-dimensional. \Cref{Bif kernels} below is a general result on the kernel equation \eqref{KerEq} from \cite{Ehrnstroem2011}, which in particular shows that it is indeed possible to choose $\Lambda \in \mcU$ such that the dimension of the kernel is one or two.
    
    \begin{lemma}[Kernel equation \cite{Ehrnstroem2011}]
        \label{Bif kernels}
        \leavevmode
        \begin{enumerate}[(i)]
            \item For every $\alpha$ and any $n$ for which $l(n,\alpha)$ is well-defined there are $\mu$ and $\lambda$ such that $n \in M(\Lambda)$.
            \item Suppose that $\lambda \in [\pi/2,\pi)$and that $n_1,n_2 \in \N_0$ satisfy
            \[
                n_2^2 \geq n_1^2 + \paren*{\frac{3\pi}{2\kappa}}^2
            \]
            Then there are $\alpha$ and $\mu$ such that $n_1, n_2 \in M(\Lambda)$ and any other solution of \eqref{KerEq} must be smaller than $n_1$.	
        \end{enumerate}
    \end{lemma}
    
    It is, however, the case that higher-dimensional kernels are, in a sense, rare\footnote{A variation of this was pointed out already in \cite{Ehrnstroem2011}; however, not taking into account the points where $l$ is not well-defined. We slightly improve upon the result here. Both $J$ and its closure are small.}:
    
    \begin{lemma}
        \label{lemma:higher_dimensional_rare}
        Let $J$ be set of all values of $\alpha$ for which there exist $\lambda$ and $\mu$ such that $\abs{M(\Lambda)} \geq 2$. Then the limit points of $J$ are contained in the set
        \begin{equation}
            \label{eq:bad_alphas}
            \{-(m_1^2\kappa^2  + m_2^2\pi^2) : m_1 \in \N_0, m_2 \in \N\}.
        \end{equation}
        In particular, $J$ consists of isolated points, except possibly those that lie in the set defined in  \eqref{eq:bad_alphas}, and has countable closure.         
    \end{lemma}
    \begin{proof} 
        The set defined in \eqref{eq:bad_alphas} consists precisely of the values of $\alpha$ for which there is at least one $n \in \N_0$ such that $l(n,\alpha)$ is not well-defined. Let $\alpha$ be such that it is \emph{not} in this set.
        
        Suppose, to the contrary, that there is a sequence $(\alpha_i)_{i \in \N}$ satisfying $\alpha_i \neq \alpha$ for all $i \in \N$ and which converges to $\alpha$, with corresponding sequences $(n_{1,i})_{i \in \N}$ and $(n_{2,i})_{i \in \N}$ such that $n_{1,i} < n_{2,i}$ and
        \begin{equation*}
            \label{eq:lhs_equality}
            l(n_{1,i},\alpha_i) = l(n_{2,i},\alpha_i)
        \end{equation*}
        for all $i \in \N$. We must necessarily have
        \[
            n_{1,i} \leq \abs{\alpha_i}^{1/2}/\kappa
        \]
        for all $i \in \N$, and so the sequence $(n_{1,i})_{i \in \N}$ is bounded. The continuity of $l(n,\cdot)$ at $\alpha$ for each $n \in \N_0$ now implies that $(l(n_{1,i},\alpha_i))_{i\in\N}$, and therefore also $(l(n_{2,i},\alpha_i))_{i\in\N}$, is bounded. This, in turn, implies that $(n_{2,i})_{i \in \N}$ is bounded.
        
        By going to a subsequence, we may assume that both $(n_{1,i})_{i \in \N}$ and $(n_{2,i})_{i \in \N}$ are constant. Thus there are $n_1 < n_2 \in \N_0$ such that $l(n_1,\alpha) = l(n_2,\alpha)$ and
        \begin{align*}
            l(n_1,\alpha_i) = l(n_2,\alpha_i)
        \end{align*}
        for all $i \in \N$. But this is impossible, because $l(n_2,\cdot)-l(n_1,\cdot)$ is a nonconstant analytic function in a neighborhood of $\alpha$.
    \end{proof}

    For later use, we give some explicit examples of one- and two-dimensional kernels for $\mcL(\Lambda)$. All satisfy $r(\Lambda)=1$, and the two-dimensional examples have been chosen such that $\theta(n_2,\alpha) =0$. To simplify the parameters involved, we choose specific values of $\kappa$.
    \begin{example}[Explicit kernels]
        \label{example:explicit_kernels}
        Let $\sigma$ be the smallest positive solution of the equation $x\cot(x)=1$.
        \begin{enumerate}[(i)]
            \item When $\kappa = 1$, $\mu=1$, $\alpha = -1$ and $\lambda =\pi/2$, the kernel is one-dimensional, being spanned by $\phi_1(x,s)=\cos(x)s$.
            \item Let $\kappa = \sigma/\sqrt{3}$. When $\mu = 1/(2\kappa)$, $\alpha = -4\kappa^2$ and $\lambda = \pi/2$, the kernel is two-dimensional, with $M = \{1, 2\}$.
            \item Let $\kappa = \sigma/\sqrt{5}$. When $\mu = 1/(3\kappa)$, $\alpha = -9\kappa^2$ and $\lambda=\pi/2$, the kernel is two-dimensional, with $M = \{2,3\}$.
        \end{enumerate}
    \end{example}

    \subsection{Arbitrarily large kernels}
        We now address a question that was raised in \cite{Ehrnstroem2011}: Do there exist $\Lambda \in \mcU$ such that $\ker{D_w\mcF(0,\Lambda)}$ is at least three-dimensional? By also letting the wave number $\kappa$ vary, this question was answered in the affirmative for dimension three in \cite{Ehrnstroem2015}. In essence, their result says that many two-dimensional kernels can be modified in order to yield a three-dimensional kernel. Here, we use a different approach to find kernels of arbitrary dimension for any $\kappa$ in a set $K$ that is dense in $(0,\infty)$.
        
        For any $\alpha < 0$ and $\lambda \in (\pi/2,\pi)$, we can obtain $r(\Lambda)=0$ by choosing $\mu$ to satisfy
        \[
            \mu^2 = - \frac{2}{\abs{\alpha}^{3/2}\sin(2\lambda)},
        \]
        which reduces the kernel equation \eqref{KerEq} to finding $m \in \N$ and $n \in \N_0$ such that
        \[
            \sqrt{\abs{\alpha}-(n\kappa)^2}=\paren*{m-\frac{1}{2}}\pi,
        \]
        is satisfied. This can be written in the form
        \begin{equation}
            \label{eq:kernel_equation_arbitrary}
            \paren*{2n\frac{\kappa}\pi}^2 + (2m-1)^2 = \frac{4\abs{\alpha}}{\pi^2}.
        \end{equation}
        We first consider the case $\kappa = \pi$.
        
        \begin{lemma}[Arbitrary kernel with $\kappa = \pi$]
            \label{lemma:arbitrary_kernel_pi}
            For $\kappa = \pi$ and any $N \in \N$, there exist $\Lambda \in \mcU$ such that $\abs{M(\Lambda)}=N$ and $0 \notin M(\Lambda)$.
        \end{lemma}
        \begin{proof}
            When $\kappa=\pi$ and $\alpha = -\pi^2 H/4$ for an odd number $H \in \N$, \eqref{eq:kernel_equation_arbitrary} becomes the Diophantine equation
            \begin{equation}
                \label{eq:kernel_equation_pi}
                (2n)^2 + (2m-1)^2 = H.
            \end{equation}
            The size of the kernel then corresponds to the number of representations of $H$ as the sum of two squares. As long as $H$ is not a square number, any such representation has $n \neq 0$ (see \Cref{remark:nontrivial_solutions}).
            
            In order to conclude, we therefore need to find an odd non-square number $H$ such that $H$ has exactly $N$ representations as a sum of squares. By \cite[Theorem 3 in Chapter 2]{Grosswald1985}, this is the case for instance when $H = p^{2N-1}$ for a prime $p \in 4\N +1$.
        \end{proof}
        
        Some examples of \Cref{lemma:arbitrary_kernel_pi} are listed below, for various choices of $H$ in \eqref{eq:kernel_equation_pi}. The values of $H$ used in second and third example, which are the smallest possible, are not in the form $p^{2N-1}$. They can easily be deduced by using the general formula given in \cite{Grosswald1985}.
        
        \begin{example}[Higher-dimensional kernels]
            \leavevmode
            \begin{enumerate}[(i)]
                \item The choice $H = 5^{2\cdot 3 -1} = 3125$ yields a three-dimensional kernel, with $M = \{5,19,25\}$. However, this is not the smallest example:
                \item Since $325=6^2 + 17^2 = 10^2 +15^2 = 18^2 + 1^2$ (with no other representations), the choice $H=325$ yields a three-dimensional kernel, with $M=\{3, 5, 9\}$.
                \item Since $1105 = 4^2+33^2=12^2+31^2 =24^2+23^2=32^2+9^2$, the choice $H=1105$ yields a four-dimensional kernel, with $M=\{2, 6, 12, 16\}$.
            \end{enumerate}
        \end{example}
        
        Let $\Q_\textnormal{o}^+$ denote the set of positive rational numbers with odd numerators when reduced to lowest terms. We can then generalize \Cref{lemma:arbitrary_kernel_pi} in the following way:
        
        \begin{theorem}[Arbitrary kernel]
            \label{theorem:arbitrary_kernel}
            For $\kappa \in \pi \Q_\textnormal{o}^+$ and any $N \in \N$, there exist $\Lambda \in \mcU$ such that $\abs{M(\Lambda)}=N$ and $0 \notin M(\Lambda)$.
        \end{theorem}
        \begin{proof}
            Write $\kappa=\pi r/s$, with $r$ and $s$ coprime. When $\alpha = -\pi^2 r^2 H/4$, \eqref{eq:kernel_equation_arbitrary} becomes
            \begin{equation}
                \label{eq:kernel_equation_dense}
                r^2(2n)^2 + s^2 (2m-1)^2 = r^2 s^2 H.
            \end{equation}
            Choosing $H = p^{2N-1}$ for a prime $p \in 4\N +1$, we know that \eqref{eq:kernel_equation_pi} has exactly $N$ solutions $(\tilde{m}_j,\tilde{n}_j)$ in $\N^2$. The pairs $(m_j, n_j) \in \N^2$ defined by
            \[
                2m_j -1 = r (2\tilde{m}_j-1), \quad n_j = s \tilde{n}_j
            \]
            then solve \Cref{eq:kernel_equation_dense}.
            
            Moreover, these are the only solutions: Suppose that $(m,n)$ solves \Cref{eq:kernel_equation_dense}. Then $r \mid (2m-1)$ and $s \mid 2n$, by coprimality of $r$ and $s$. It follows that $2m-1 = r(2\tilde{m}-1)$ and $2n = s \hat{n}$, where $\tilde{m}$ and $\hat{n}$ solve $\hat{n}^2 + (2\tilde{m}-1)^2 = H$. Since $H$ is odd, $\hat{n}=2\tilde{n}$. Uniqueness in \Cref{eq:kernel_equation_pi} now yields the result.
        \end{proof}
        
        Although we provide kernels of arbitrary dimension in \Cref{theorem:arbitrary_kernel}, the corresponding triples $\Lambda$ satisfy $r(\Lambda)=0$, unlike the three-dimensional kernels obtained in \cite{Ehrnstroem2015}. In particular, this means that the two-dimensional bifurcation result in \Cref{2dBif} does not apply for the kernels from \Cref{theorem:arbitrary_kernel} with $N=2$. We remark that an obstacle for higher-dimensional bifurcation is that there are only four parameters to work with, namely $\Lambda$ and $\kappa$. This may be remedied by for instance including surface tension.
        
        An application of \Cref{theorem:arbitrary_kernel} is one-dimensional bifurcation with several different wave numbers for fixed $\Lambda$. If the set $M(\Lambda) = \{n_1,\ldots,n_N\}$ is such $n_i \nmid n_j$ for all $i\neq j$, we can make restrictions to each $X^{(n_i)}$ and then apply \Cref{1dBif}. This will yield $N$ different solution curves. Two examples for which the condition on $M$ is fulfilled are $H=725$ and $H=3145$, corresponding to $M=\{5,7,13\}$ and $M=\{18,24,26,28\}$, respectively.
    
    \subsection{Lyapunov--Schmidt reduction} \label{LSred}
        Before we can reduce \eqref{eq:F_problem} to a finite-dimensional problem by applying the Lyapunov--Schmidt reduction, we need the following result from elliptic theory.
        \begin{theorem}[Fredholm property \cite{Ehrnstroem2011}] \label{mcLisFred}
            The operator $\mcL(\Lambda)$ is Fredholm for each $\Lambda \in \mcU$, with index $0$. The range of $\mcL(\Lambda)$ is the orthogonal complement of
            \[
                Z\coloneqq \brac[\big]{ (\eta_{\phi},\phi): \phi \in \ker \mcL(\Lambda)} \subset \tilde{X} \subset Y
            \]
            in $Y$ with respect to the inner product
            \begin{equation}
                \label{eq:inner_product_Y}
                \ip{w_1}{w_2}_Y =  \smashoperator{\int_{0}^{2 \pi/\kappa}} \eta_1 \eta_2 \dx + \int\limits_0^1\,\smashoperator{\int_0^{2 \pi/\kappa}} \hat{\phi}_1 \hat{\phi}_2 \dx \ds, \quad w_j \in Y.
            \end{equation}
            Let $\tilde w_n \coloneqq (\eta_{\phi_n}, \phi_n)$ for $n \in M(\Lambda)$, where $\phi_n$ and $M(\Lambda)$ are as in \Cref{kernelLemma}. Then the projection $\Pi_Z:Y \to Z$ onto $Z$ along $\ran \mcL(\Lambda)$ is given by
            \begin{equation}
                \label{PiZ}
                \Pi_Z w = \sum_{n \in M(\Lambda)} \frac{\ip{w}{\tilde w_n}_Y}{\norm{\tilde w_n}_Y^2} \tilde w_n.
            \end{equation}
        \end{theorem}
        \begin{remark}
            \label{rem:fredholm}
            That $D_w \mcF(0,\Lambda)$ be Fredholm is the main reason for making the assumption \eqref{psiAssumption}. When we have equality in \eqref{psiAssumption}, \eqref{DF1w} reduces to $D_w \mcF_1(0,\Lambda)w = \eta$, whence $\ran{D_w\mcF_1(0,\Lambda)=X_1}$. The operator $D_w \mcF(0,\Lambda)$ then cannot be Fredholm, since $X_1$ is not closed in $Y_1$.
        \end{remark}
        
        Let $\Lambda^* \in \mcU$ be a triple $(\mu^*,\alpha^*,\lambda^*)$ such that $N \coloneqq \abs{M(\Lambda^*)}\geq 1$. Then \Cref{kernelLemma} says that the pairs
        \[
            w_n^* \coloneqq \mcT(\Lambda^*)\phi_n^* \in X, \quad n \in M(\Lambda^*),
        \]
        span the kernel of $D_w\mcF(0,\Lambda^*)$. Since the kernel is finite-dimensional, there exists a closed subspace $X_0 \subset X$ such that
        \[
            X = \ker D_w\mcF(0,\Lambda^*) \oplus X_0.
        \]
        
        By \Cref{mcLisFred}, we can also decompose $Y$ into the direct sum
        \[
            Y = Z \oplus \ran \mcL(\Lambda^*),
        \]
        where $Z \coloneqq \lspan{\{ \tilde{w}_n^* \}_{n \in M(\Lambda^*)}}$, since these are orthogonal complements in the inner product \eqref{eq:inner_product_Y} on $Y$. Applying the Lyapunov--Schmidt reduction (see e.g. Kielhöfer \cite{Kielhoefer2012}) for these decompositions of $X$ and $Y$, we obtain the following lemma.
        \begin{lemma}[Lyapunov--Schmidt]
            \label{LyaSch}
            There exist open neighborhoods $\mc{N}$ of $0$ in $\ker D_w \mcF(0,\Lambda^*)$, $\mc{M}$ of $0$ in $X_0$, and $\mcU'$ of $\Lambda^*$ in $\mcU$, and a uniquely determined function $\psi: \mc{N} \times \mcU' \to  \mc{M}$ such that
            \[
                \mcF(w,\Lambda) = 0 \qquad \text{for } w \in \mc{N}+\mc{M}, \, \Lambda \in \mcU',
            \]
            if and only if $w = w^*+\psi(w^*,\Lambda)$ and $w^* = \sum_{n \in M(\Lambda^*)} t_n w_n^* \in \mc{N}$ solves the finite-dimensional problem
            \[
                \Phi(t,\Lambda) = 0 \qquad \text{for } t \in \mc{V}, \, \Lambda \in \mcU',
            \]
            where
            \[
                \Phi(t,\Lambda) \coloneqq \Pi_Z \mcF(w,\Lambda) \quad \text{and} \quad \mc{V} \coloneqq \brac*{ (t_n)_{n \in M(\Lambda^*)} \in \R^N : w^* \in \mc{N}}.
            \]
            
            The function $\psi$ is smooth, and satisfies $\psi(0,\Lambda) = 0$ for all $\Lambda \in \mcU'$, and $D_w \psi(0,\Lambda^*) = 0$.
        \end{lemma}

\section{One-dimensional bifurcation} \label{1dBifSection}
    We are now in a position to show that a curve of nontrivial solutions of \eqref{eq:F_problem} bifurcates from each point $(0,\Lambda^*) \in X \times \mcU$ where the kernel of $D_w \mcF(0,\Lambda^*)$ is one-dimensional, given that $\Lambda^*$ satisfies an additional technical condition. This condition comes from \Cref{IpCalcLemmalambda} below.
    
    \begin{lemma}[Orthogonality]
        \label{IpCalcLemmalambda}
        Suppose that $n \in M(\Lambda)$, so that the function $\phi_n$ given by \eqref{phik} lies in $\ker \mcL(\Lambda)$. Then, if $\tilde w_n \coloneqq (\eta_{\phi_n},\phi_n)$ is the corresponding basis function of $Z$, we have
        \begin{equation}
            \label{eq:transversality_inner_product}
            \ip{D_{\lambda} \mcL(\Lambda) \phi_n}{\tilde{w}_n}_Y = A \paren*{ \frac{\sinh(\theta_n)}{\theta_n}}^2,
        \end{equation}
        where
        \[
            A \coloneqq -\frac{2\pi}{\kappa \psi_{0s}(1)^2} \brak[\bigg]{\cot(\lambda) + \frac{\mu^2 |\alpha|^{3/2}}{2}}
        \]
        does not depend on $n$. In particular,
        \[
            \ip{D_{\lambda} \mcL(\Lambda) \phi_n}{\tilde{w}_n}_Y = 0 \quad \text{if and only if}\quad \cot(\lambda) = -\frac{\mu^2 |\alpha|^{3/2}}{2}.
        \]
    \end{lemma}
    \begin{proof}
        Recalling \eqref{mcLExp}, we find the derivative
        \[
            D_\lambda \mcL(\Lambda)\phi = \paren*{\psi_{0s\lambda}(1)\phi_s|_{s=1} - \paren*{\psi_{0ss\lambda}(1) - \frac{\psi_{0s\lambda}(1)}{\psi_{0s}(1)^2}}\phi|_{s=1},0}.
        \]
        Using that $\phi_n(x,s) = \cos(n \kappa x) \sinh(\theta_n s)/\theta_n$, we get
        \[
            D_\lambda \mcL_1(\Lambda)\phi_n =\tilde{A} \frac{\sinh(\theta_n)}{\theta_n} \cos(n\kappa x),
        \]
        where
        \begin{align*}
            \tilde{A}&\coloneqq\psi_{0s\lambda}(1) l(n,\alpha) - \paren*{ \psi_{0ss\lambda}(1) - \frac{\psi_{0s\lambda}(1)}{\psi_{0s}(1)^2}}\\
            &=\frac{2}{\psi_{0s}(1)}\paren[\bigg]{\cot(\lambda) + \frac{\mu^2 \abs{\alpha}^{3/2}}{2}},
        \end{align*}
        by the kernel equation \eqref{KerEq} and the definition of $\psi_0$.
        
        Since $\tilde w_n = (\eta_{\phi_n},\phi_n)$, with
        \[
            \eta_{\phi_n}(x) = -\frac{\phi_n(x,1)}{\psi_{0s}(1)} = -\frac{1}{\psi_{0s}(1)} \frac{\sinh(\theta_n)}{\theta_n}\cos(n\kappa x),
        \]
        we now find
        \begin{align*}
            \ip{D_{\lambda} \mcL(\Lambda) \phi_n}{\tilde{w}_n}_Y &=\smashoperator{\int_{0}^{2 \pi/\kappa}} \eta_{\phi_n} D_{\lambda} \mcL_1(\Lambda)\phi_n\dx\\
            &= -\frac{\pi\tilde{A}}{\kappa\psi_{0s}(1)}\paren*{ \frac{\sinh(\theta_n)}{\theta_n}}^2,
        \end{align*}
        which is \eqref{eq:transversality_inner_product} with $A = -\pi \tilde{A}/(\kappa \psi_{0s}(1))$.
    \end{proof}
    
    We will refer to
    \begin{equation} \label{lambdaNuisance}
        \cot(\lambda) \neq - \frac{\mu^2 |\alpha|^{3/2}}{2}
    \end{equation}
    as \textit{the transversality condition}, because it corresponds to transversality in the Crandall--Rabinowitz theorem (see \cite{Crandall1971} or \cite{Kielhoefer2012}). Note that all the examples we provided in \Cref{example:explicit_kernels} satisfy this condition. It is straightforward to check that the transversality condition fails at $\Lambda^* \in \mcU$ precisely when $\mu_*'(\lambda^*)=0$ in \Cref{theorem:kernel_local_description}. This means that we can obtain the following by moving slightly along the graph of $\mu_*$ (see \Cref{fig:kernel_geometry}).
    
    \begin{lemma} \label{OrthogonalityLemma}
        Suppose that $\Lambda^* = (\mu^*, \alpha^*, \lambda^*) \in \mcU$ is such that the transversality condition \eqref{lambdaNuisance} fails. Then there are $\mu, \lambda \in \R$ with $\Lambda = (\mu,\alpha^*, \lambda) \in \mcU$ such that the transversality condition holds and $M(\Lambda) = M(\Lambda^*)$. The triple $\Lambda$ can be chosen arbitrarily close to $\Lambda^*$.
    \end{lemma}
    
    The one-dimensional bifurcation result is an application of the Crandall--Rabinowitz bifurcation theorem. To clarify the proof of the two-dimensional bifurcation in the next section, we will nonetheless spell out the details of the proof.
    \begin{theorem}[One-dimensional bifurcation] \label{1dBif}
        Suppose that $\Lambda^* \in \mcU$ is such that $M(\Lambda^*) = \{n\}$ with $n \in \N$, and therefore that
        \[
            \ker{D_w\mcF(0,\Lambda^*) = \lspan{\{w^*\}}},
        \]
        where $w^* = \mathcal{T}(\Lambda^*) \phi^*$, with $\phi^*\coloneqq \phi_n$ as in \Cref{kernelLemma}. If the transversality condition \eqref{lambdaNuisance} holds, there exists a smooth curve $\{ (\overline{w}(t),\overline{\lambda}(t)): 0 < |t| < \varepsilon \}$ of nontrivial small-amplitude solutions to
        \begin{equation}
            \label{FEq}
            \mcF(w,\mu^*,\alpha^*,\lambda) = 0,
        \end{equation}
        in $\mcO \times (0,\pi)$, passing through $(\overline{w}(0),\overline{\lambda}(0)) = (0, \lambda^*)$, with
        \begin{equation}
            \label{eq:1d_asymptotic}
            \overline{w}(t) = t w^* + O(t^2) \quad \text{in } X \text{ as } t \to 0.
        \end{equation}
        These are all the nontrivial solutions of \eqref{FEq} in a neighborhood of $(0,\lambda^*)$ in $\mcO \times (0,\pi)$.
    \end{theorem}
    
    \begin{proof}
        Using \Cref{LyaSch}, we know that there exists a neighborhood of $(0,\lambda^*)$ in $\mcO \times (0,\pi)$ for which the equation $\mcF(w,\mu^*,\alpha^*,\lambda) = 0$ is equivalent to $\Phi(t,\mu^*,\alpha^*,\lambda) = 0$, where $t \in \R$. From the same lemma we also have the identity $\Phi(0,\Lambda)=0$, and hence we can write
        \[
            \Phi(t,\Lambda) = \int_0^1 \partial_z \paren[\big]{\Phi(tz,\Lambda)}\dz=t\Psi(t,\Lambda),
        \]
        where
        \begin{equation}
            \label{eq:Psi_definition}
            \Psi(t,\Lambda) \coloneqq \int_0^1 \! \Phi_{t}(tz,\Lambda)\dz
        \end{equation}
        is smooth. For nontrivial solutions ($t \neq 0$), the equations $\Phi=0$ and $\Psi = 0$ are equivalent, whence we need only concern ourselves with the latter equation.
        
        We want to apply the implicit function theorem to $\Psi$, which requires that $\Psi(0,\Lambda^*) = 0$ and $\Psi_\lambda (0,\Lambda^*) \neq 0$ (recall that $Z$ is one-dimensional). Now, from \eqref{eq:Psi_definition}, we find
        \begin{align*}
            \Psi(0,\Lambda^*) &= \Phi_t(0,\Lambda^*),\\
            \Psi_\lambda(0,\Lambda^*) &=  \Phi_{t\lambda}(0,\Lambda^*),
        \end{align*}
        so these are the derivatives of $\Phi$ we need to compute. By the definition of $\Phi$,
        \begin{equation}
            \label{eq:Phi_t}
            \Phi_t(t,\Lambda) = \Pi_Z  D_w \mcF(t w^*+\psi(t w^*, \Lambda),\Lambda) (w^*+D_w\psi(tw^*,\Lambda) w^*),
        \end{equation}
        and so by evaluating in $t=0$, and using the properties of $\psi$ listed in \Cref{LyaSch}, we have
        \[
            \Phi_t(0,\Lambda) = \Pi_Z  D_w \mcF(0,\Lambda) (w^*+D_w\psi(0,\Lambda) w^*),
        \]
        which also yields
        \begin{multline*}
            \Phi_{t\lambda}(0,\Lambda) = \Pi_Z  D_{w\lambda} \mcF(0,\Lambda) (w^*+D_w\psi(tw^*,\Lambda))\\ + \Pi_Z D_w \mcF(0,\Lambda) D_{w\lambda}\psi (w^*,\Lambda).
        \end{multline*}
        
        We now obtain
        \[
            \Psi(0,\Lambda^*)=\Phi_t(0,\Lambda^*) = \Pi_Z D_w \mcF(0,\Lambda^*)w^* = 0,
        \]
        because $D_w\psi(0,\Lambda^*)=0$ by the last part of \Cref{LyaSch}, and because $\Pi_Z$ projects along $\ran D_w\mcF(0,\Lambda^*)$. Similarly,
        \begin{equation}
            \label{eq:Phi_t_lambda}
            \Phi_{t\lambda}(0,\Lambda^*) = \Pi_Z D_{w\lambda} \mcF(0,\Lambda^*)w^*.
        \end{equation}
        
        Note that $D_w \mcF(0,\Lambda)w^* = \mcL(\Lambda)\mcT(\Lambda)^{-1} w^*$, and hence
        \[
            D_{w\lambda}\mcF(0,\Lambda^*)w^* = D_{\lambda} \mcL(\Lambda^*) \phi^* - D_w \mcF(0,\Lambda^*) \partial_\lambda \mcT(\Lambda^*) \phi^*,
        \]
        which implies that \Cref{eq:Phi_t_lambda} can be written
        \[
            \Phi_{t\lambda}(0,\Lambda^*) = \Pi_Z D_\lambda \mcL(\Lambda^*)\phi^*,
        \]
        again using that $\Pi_Z$ projects along $\ran D_w\mcF(0,\Lambda^*)$. We can now use \Cref{IpCalcLemmalambda} to deduce that $\Psi_\lambda(0,\Lambda^*) = \Phi_{t\lambda}(0,\Lambda^*) \neq 0$, due to the assumption of transversality.
        
        Finally, since $\Psi(0,\Lambda^*)=0$ and $\Psi_\lambda(0,\Lambda^*) \neq 0$, we can invoke the implicit function theorem to deduce that there exists an $\varepsilon>0$ and a smooth function $\map{\overline \lambda}{(-\varepsilon,\varepsilon)}{\R}$ with $\overline \lambda(0) = \lambda^*$ such that $\Psi(t,\mu^*,\alpha^*,\overline{\lambda}(t)) \equiv 0$. Moreover, the curve $\{(t,\overline{\lambda}(t)) : \abs{t} < \varepsilon\}$ describes all solutions to $\Psi(t,\mu^*,\alpha^*,\lambda) = 0$ in a neighborhood of $(0,\lambda^*)$. The corresponding solution curve to $\mcF(w,\mu^*,\alpha^*,\lambda)=0$ is $\{(\overline{w}(t),\overline{\lambda}(t)) : \abs{t} < \varepsilon\}$, where $\overline{w}(t) \coloneqq t w^*+\psi(t w^*,\mu^*,\alpha^*,\overline \lambda(t))$. It follows that
        \[
            \dot{\overline{w}}(t) = w^* + D_w\psi(tw^*,\mu^*,\alpha^*,\overline \lambda(t))w^* + D_\lambda\psi(tw^*,\mu^*,\alpha^*,\overline\lambda(t)) \dot{\overline \lambda}(t),
        \]
        and we can conclude, once again using the properties of $\psi$ given in \Cref{LyaSch}, that $\overline{w}(0) = 0$ and $\dot{\overline{w}}(0) = w^*$. Consequently, we obtain \eqref{eq:1d_asymptotic}.
    \end{proof}
    
    If $\varepsilon$ is sufficiently small, the waves obtained from \Cref{1dBif} are Stokes waves. This can be seen from the asymptotic formula in \eqref{eq:1d_asymptotic}.

    \subsection{Properties of the bifurcation curve}
        \label{1dTaylorSec}
        
        The one-dimensional bifurcation result in \Cref{1dBif} is analogous to \cite[Theorem 4.6]{Ehrnstroem2011}, which uses $\mu$ instead of $\lambda$ as the bifurcation parameter. Other than the parameters, the main difference between the theorems is the addition of the transversality condition \eqref{lambdaNuisance} for bifurcation with respect to $\lambda$. Here, we will investigate the properties of the solution curves more closely.
        
        The motivation is to understand the solution set of \eqref{eq:F_problem} better, and in particular to rule out the possibility that the solution curve found here coincides with the one from \cite{Ehrnstroem2011}. The only way this can occur is if $\overline{\lambda}(t)$ and $\overline{\mu}(t)$, in the notation of \cite{Ehrnstroem2011}, are constant along the curves. (If they were constant, we would obtain the same solutions by uniqueness in \Cref{1dBif}.) \Cref{lambdaDotZero} shows that we need to consider at least second-order properties of the bifurcation curve in order to achieve this.
        
        \begin{proposition}[First derivative of $\overline{\lambda}$] \label{lambdaDotZero}
            Under the hypothesis of \Cref{1dBif}, the function $\overline{\lambda}$ satisfies
            \[
                \dot{\overline{\lambda}}(0) = 0,
            \]
            and so the bifurcation parameter is constant to the first order along the bifurcation curve.
        \end{proposition}
        \begin{proof}
            We adopt the notation used in the proof of \Cref{1dBif}. Differentiation of the identity $\Psi(t,\mu^*,\alpha^*,\overline{\lambda}(t)) = 0$, and evaluation at $t=0$, yields the equation
            \begin{equation}
                \label{eq:lambda_dot_equation}
                \Psi_t(0,\Lambda^*) + \Psi_\lambda (0,\Lambda^*)\dot{\overline{\lambda}}(0)=0
            \end{equation}
            for the derivative of $\overline{\lambda}$ at the origin. From \eqref{eq:Phi_t_lambda} and the discussion immediately after, we know that
            \begin{equation}
                \label{eq:Psi_lambda}
                \Psi_\lambda(0,\Lambda^*) = \Pi_ZD_\lambda \mcL(\Lambda^*)\phi^* \neq 0,
            \end{equation}
            which means that \eqref{eq:lambda_dot_equation} uniquely determines $\dot{\overline{\lambda}}(0)$. However, we still need to compute $\Psi_t(0,\Lambda^*)$.
            
            From \eqref{eq:Psi_definition}, we obtain
            \[
                \Psi_t(0,\Lambda^*) = \frac{1}{2}\Phi_{tt}(0,\Lambda^*),
            \]
            and differentiation in \Cref{eq:Phi_t} leads to
            \[
                \Phi_{tt}(t,\Lambda) = \begin{multlined}[t]\Pi_Z D_{w}^2 \mcF(tw^*+\psi(tw^*,\Lambda),\Lambda)(w^*+D_w\psi(tw^*,\Lambda)w^*)^2 \\+ \Pi_Z D_w \mcF(tw^*+\psi(tw^*,\Lambda),\Lambda)D_w^2\psi(tw^*,\Lambda)(w^*)^2.\end{multlined}
            \]
            Hence, by using the properties of $\psi$ given in \Cref{LyaSch}, and using that $\Pi_Z$ projects along the range of $D_w \mcF(0,\Lambda^*)$, we find
            \begin{equation}
                \label{eq:Psi_t}
                \Psi_t(0,\Lambda^*) = \frac{1}{2}\Pi_ZD_w^2\mcF(0,\Lambda^*)(w^*)^2.
            \end{equation}
            
            Using \Cref{eq:lambda_dot_equation,eq:Psi_lambda,eq:Psi_t} and the formula \eqref{PiZ} for $\Pi_Z$ given in \Cref{mcLisFred}, we now find
            \begin{equation}
                \label{eq:lambda_dot}
                \dot{\overline{\lambda}}(0) = - \frac{\ip*{D_w^2\mcF(0,\Lambda^*)(w^*)^2}{\tilde{w}^*}_Y}{2\ip{D_\lambda \mcL(\Lambda^*)\phi^*}{\tilde{w}^*}_Y},
            \end{equation}
            and so it is sufficient to show that the numerator,
            \begin{equation}
                \label{eq:numerator_integrals}
                \smashoperator{\int_{0}^{2 \pi/\kappa}} \eta^*D_w^2 \mcF_1(0,\Lambda) (w^*)^2 \dx + \int\limits_0^1\,\smashoperator{ \int_0^{2\pi/\kappa}} \phi^* D_w^2 \mcF_2(0,\Lambda) (w^*)^2 \dx \ds, 
            \end{equation}
            vanishes. Since $w^* = \mcT(\Lambda^*)\phi^*$ with $\phi^*$ being a separable function of $x$ and $s$, so is $\eta^*$. Moreover, we see from \Cref{phik} that their $x$-dependence is through $\cos(n\kappa x)$. Thus each term in $ D_w^2 \mcF_1(0,\Lambda^*) (w^*)^2$ and $ D_w^2\mcF_1(0,\Lambda^*) (w^*)^2$ has an $x$-dependence of the form $\sin^a(n\kappa x) \cos^b( n\kappa x)$ with $a+b=2$ (see the derivatives listed in \Cref{appendix}). It follows that we will be integrating terms whose $x$-dependence is $\sin^a(n \kappa x) \cos^b(n \kappa x)$ with $a+b=3$ in \eqref{eq:numerator_integrals}, and therefore that the numerator in \eqref{eq:lambda_dot} vanishes.
        \end{proof}
        \begin{remark}
            \Cref{eq:lambda_dot} also holds if $\mu$ is substituted for $\lambda$. Since the proof of \Cref{lambdaDotZero} only depends on the fact that the numerator in \eqref{eq:lambda_dot} vanishes, we can conclude that one also has $\dot{\overline{\mu}}(0)=0$ when using $\mu$ as the bifurcation parameter.
        \end{remark}
        
        To consider the question of second-order behavior of the solution curve, let us return to the expression for $\dot{\overline{w}}$ we found in the one-dimensional bifurcation result \Cref{1dBif}, namely
        \begin{equation}
            \label{eq:w_dot}
            \dot{\overline{w}}(t) = w^* + D_w\psi w^* + \dot{\overline{\lambda}}(t)\psi_\lambda ,
        \end{equation}
        where $D_w\psi$ and $\psi_\lambda$ are evaluated at $(tw^*,\mu^*,\alpha^*,\overline{\lambda}(t))$. Taking another derivative in \eqref{eq:w_dot} yields
        \begin{equation}
            \label{eq:w_dot_dot}
            \ddot{\overline{w}}(t) = D_w^2\psi (w^*)^2+2\dot{\overline\lambda}(t)D_{w \lambda}\psi w^*  +\dot{\overline\lambda}(t)^2\psi_{\lambda \lambda} +\ddot{\overline \lambda}(t)\psi_\lambda .
        \end{equation}
        This simplifies significantly at $t=0$. An expression for $\ddot{\overline{\lambda}}(0)$ can also be found, akin to how \eqref{eq:lambda_dot} was derived. The details are omitted, see for instance \cite[Section I.6]{Kielhoefer2012}.
        
        \begin{proposition}[Second derivatives]
            \label{prop:second_derivatives}
            Under the hypothesis of \Cref{1dBif}, we have
            \[
                \ddot{\overline{w}}(0) = D_w^2\psi(0,\Lambda^*) (w^*)^2.
            \]
            and
            \[
                \ddot{\overline{\lambda}}(0) = - \frac{\ip*{D_{w}^3\mcF(0,\Lambda^*)(w^*)^3+3D_{w}^2\mcF(0,\Lambda^*)(w^*,D_{w}^2\psi(0,\Lambda^*)(w^*)^2)}{\tilde{w}^*}_Y}{3\ip{D_{\lambda}\mcL(\Lambda^*)\phi^*}{\tilde{w}^*}_Y}
            \]
            for the solution curve $(\overline{w}(t),\overline{\lambda}(t))$.
            
        \end{proposition}
        \begin{proof}[Proof for $\ddot{\overline{w}}(0)$]
            In view of \Cref{LyaSch}, we have that $\psi(0,\Lambda) = 0$ for all $\Lambda$ in an open neighborhood of $\Lambda^*$, and thus $\psi_\lambda(0,\Lambda^*)$ is zero. Finally, the second and third terms in \eqref{eq:w_dot_dot} vanish due to \Cref{lambdaDotZero}.
        \end{proof}
        \begin{remark}
            The expression for $\ddot{\overline{\mu}}(0)$ can be obtained by simply substituting $\mu$ for $\lambda$ in the expression for $\ddot{\overline{\lambda}}(0)$. From this, it follows that
            \[
                \ddot{\overline{\lambda}}(0) = \frac{\ip{D_\mu\mcL(\Lambda^*)\phi^*}{\tilde{w}^*}_Y}{\ip{D_\lambda\mcL(\Lambda^*)\phi^*}{\tilde{w}^*}_Y}\ddot{\overline{\mu}}(0)= \frac{\ddot{\overline{\mu}}(0)}{\mu(\cot(\lambda)+\mu^2 \abs{\alpha}^{3/2}/2)}.
            \]
             In particular, this implies that $\ddot{\overline{\lambda}}(0)$ and $\ddot{\overline{\mu}}(0)$ must have either the same or opposite sign, depending on the sign of $\mu$ and which ``side'' of the transversality condition \eqref{lambdaNuisance} $\Lambda^*$ is on.
        \end{remark}
        
        We now give a more transparent description of $D_w^2\psi(0,\Lambda^*)(w^*)^2$, which \Cref{prop:second_derivatives} shows is required for computing both $\ddot{\overline{w}}(0)$ and $\ddot{\overline{\lambda}}(0)$.
        \begin{lemma}[Description of $D_w^2\psi(0,\Lambda^*)(w^*)^2$]
            \label{lem:psi_second_derivative}
            Write
            \begin{equation}
                \label{eq:Dww_F1_F2}
                \begin{aligned}
                    D_w^2\mcF_1(0,\Lambda^*)(w^*)^2 &= c_0 + c_2\cos(2n\kappa x),\\
                    D_w^2\mcF_2(0,\Lambda^*)(w^*)^2 &= b_0(s) + b_2(s)\cos(2n\kappa x),
                \end{aligned}
            \end{equation}
            and let $\zeta \in C_{\kappa,\textnormal{e}}^{2,\beta}(\overline{\hat{\Omega}})$ be such that
            \[
                \zeta(x,s) \coloneqq a_0(s) + a_2(s)\cos(2n\kappa x),
            \]
            where the coefficients $a_0$ and $a_2$ solve the boundary value problems
            \begin{equation}
                \label{eq:a0_a2_equations}
                \begin{gathered}
                    a_j''(s) - \theta_{jn}^2 a_j(s) = -b_j(s),\\
                    a_j(0)=0, \quad \psi_{0s}(1)a_j'(1) - \paren*{\psi_{0ss}(1) + \frac{1}{\psi_{0s}(1)}}a_j(1)=-c_j,
                \end{gathered}
            \end{equation}
            for $j=0,2$. Then
            \[
                D_w^2\psi(0,\Lambda^*)(w^*)^2 = \mcT(\Lambda^*)\zeta.
            \]
        \end{lemma}
        \begin{proof}
            That $D_w^2\mcF(0,\Lambda^*)(w^*)^2$ can always be written as in \eqref{eq:Dww_F1_F2} can be deduced from the expressions for the derivatives of $\mcF$ listed in \Cref{appendix}.
            
            The function $\psi$ satisfies the identity
            \[
                (I-\Pi_Z)\mcF(tw^* + \psi(tw^*,\Lambda^*),\Lambda^*)= 0
            \]
            for sufficiently small $t$. If we take two derivatives of this equation and evaluate at $t=0$ we obtain the equation
            \[
                (I-\Pi_Z)(D_w^2\mcF(0,\Lambda^*)(w^*)^2 + D_w\mcF(0,\Lambda^*)D_w^2\psi(0,\Lambda^*)(w^*)^2)=0
            \]
            for $D_w^2\psi(0,\Lambda^*)(w^*)^2$. Since we established in the proof of \Cref{lambdaDotZero} that $D_w^2 \mcF(0,\Lambda^*)(w^*)^2$ lies in the range of $D_w \mcF(0,\Lambda^*)$, and since $\Pi_Z$ projects along $\ran D_w \mcF(0,\Lambda^*)$, this implies that
            \begin{equation}
                \label{psiwwEq}
                D_w \mcF(0,\Lambda^*) D_w^2\psi(0,\Lambda^*)(w^*)^2 = - D_w^2 \mcF(0,\Lambda^*)(w^*)^2,
            \end{equation}
            which uniquely determines $D_w^2\psi(0,\Lambda^*)(w^*)^2$.
            
            If we now introduce the function $\zeta \in \tilde X_2$ by
            \[
                \mcT(\Lambda^*)\zeta \coloneqq D_w^2\psi (0,\Lambda^*)(w^*)^2,
            \]
            then \Cref{psiwwEq} can be written
            \[
                \mcL(\Lambda^*) \zeta  = - D_w^2 \mcF(0,\Lambda^*)(w^*)^2.
            \]
            Utilizing \eqref{mcLExp}, this proves the lemma. 
        \end{proof}
        
        Due to the form of $w^*$ and the expressions for the derivatives of $\mcF$ in \Cref{appendix}, we know that the coefficients $a_0$ and $a_2$ in \Cref{lem:psi_second_derivative} are polynomials in $s$, $\sinh(\theta_{jn}s)/\theta_{jn}$ and $\cosh(\theta_{jn}s)$ for $j=0,1,2$. They can, with some effort, be computed explicitly using a computer algebra system. However, the general expressions are much too long to perform any useful analysis of the second derivatives. We will therefore content ourselves with presenting the result for the first special case of \Cref{example:explicit_kernels}, which was constructed specifically to make $\phi^*$ and $\psi_0$ as simple as possible. This, in turn, yields particularly simple $a_0$ and $a_2$.
        
        \begin{theorem}[Special case]
            \label{thm:special_case}
            When $\kappa = 1$ and $\Lambda^* = (1,-1,\pi/2)$, the functions $a_0$ and $a_2$ are given by
            \begin{equation}
                \label{eq:a0_a2_special_case}
                \begin{aligned}
                    a_0(s) &= s + \frac{1}{2}s^2 \sin(s-1) + \frac{3\sin(s)}{2(\cos(1)-\sin(1))},\\
                    a_2(s) &= s + \frac{1}{2}s^2\sin(s-1) + \frac{\sinh(\sqrt{3}s)}{2(\sqrt{3}\cosh(\sqrt{3}) -\sinh(\sqrt{3}))},
                \end{aligned}
            \end{equation}
            respectively. This yields
            \begin{equation}
                \label{eq:lambda_dot_dot_special_case}
                \ddot{\overline{\lambda}}(0) = \frac{3}{2} + 3a_0(1) + \frac{1}{2}a_2(1) < 0.
            \end{equation}
            In particular, $\overline{\lambda}$ is not constant along the bifurcation curve, which therefore does not coincide with the one found in \cite{Ehrnstroem2011}.
        \end{theorem}
        \begin{proof}
            The kernel of $\mcL(\Lambda^*)$ is spanned by $\cos(x)s$, and moreover $\psi_0(s) = -\sin(s-1)$. One may check that for this special case, the coefficients in \Cref{eq:Dww_F1_F2} are given by
            \begin{align*}
                c_0 &= 2, & b_0(s)&= -s -2s\cos(s-1)-\sin(s-1),\\
                c_2 &= 1, & b_2(s)&= 3s-2s\cos(s-1) +(2s^2-1)\sin(s-1).
            \end{align*}
            It follows by direct verification that the functions $a_0$ and $a_2$ in \eqref{eq:a0_a2_special_case} solve the boundary value problems in \eqref{eq:a0_a2_equations}. Finally, a long (but direct) computation from the expression for $\ddot{\overline{\lambda}}(0)$ in \Cref{prop:second_derivatives} yields \eqref{eq:lambda_dot_dot_special_case}.
        \end{proof}
        
        \begin{figure}[htb]
            \centering
            \includegraphics{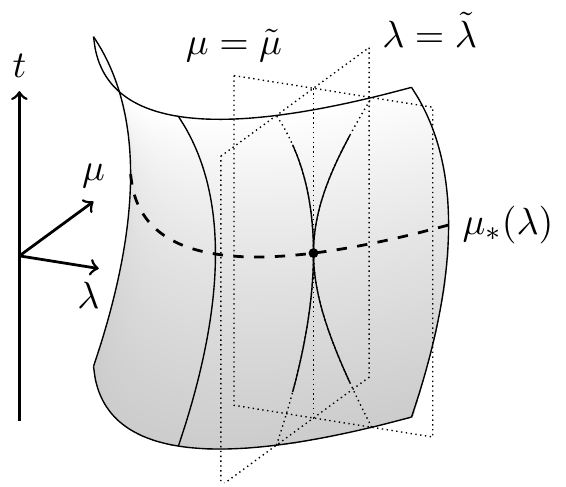}
            \caption{The solution curves emanating from the graph of $\mu_*$, making up the solution set of $\Psi=0$ when $\alpha$ is fixed. The specific point used in \Cref{thm:special_case} can be found to the left of where the transversality condition fails.}
            \label{fig:1d_solutions_geometry}
        \end{figure}
        
        For the same special case as in \Cref{thm:special_case}, we can consider bifurcation from other points on the graph of the associated function $\mu_*$ that was introduced in \Cref{theorem:kernel_local_description}. One may verify that the numerator in the expression for $\ddot{\overline{\lambda}}(0)$ in \Cref{prop:second_derivatives} is negative on the entire graph of $\mu_*$. This means that, locally, the solution set of the function $\Psi$ from the proof of \Cref{1dBif} looks qualitatively like the surface shown in \Cref{fig:1d_solutions_geometry}, when $\alpha$ is fixed. Recall that the transversality condition corresponds to $\mu_*' \neq 0$, and observe that $\ddot{\overline{\lambda}}(0)$ changes sign when this condition fails.
        
        We remark that for some other choices of $\kappa$ and $\Lambda$ the numerator \emph{does} change sign on the graph of $\mu_*$. It follows that \Cref{fig:1d_solutions_geometry} does not, in general, tell the whole story.

    \subsection{Local description}
        \label{LocClass1D}          
        Using \cite[Theorem 4.6]{Ehrnstroem2011}, or \Cref{1dBif} when the transversality condition \eqref{lambdaNuisance} is fulfilled, we can describe all solutions of \Cref{eq:F_problem} in a neighborhood of any $(0,\Lambda^*)$ in $X \times \mcU$ for which $\abs{M(\Lambda^*)}=1$.
        
        Suppose that we have such a point, and that $M(\Lambda^*) = \{n\}$. Then \Cref{theorem:kernel_local_description} tells us that there is a neighborhood of $\Lambda^*$ in which
        \[
            M(\Lambda) = \begin{cases}
                \{n\} & \mu = \mu_1(\alpha,\lambda),\\
                \varnothing & \text{otherwise}.
            \end{cases}
        \]
        This allows us to invoke \cite[Theorem 4.6]{Ehrnstroem2011} on each point on the graph of $\mu_1$, obtaining a family of solution curves. These are, in fact, all the nontrivial solutions near $(0,\Lambda^*)$:
        
        \begin{figure}[htb]
            \centering
            \includegraphics{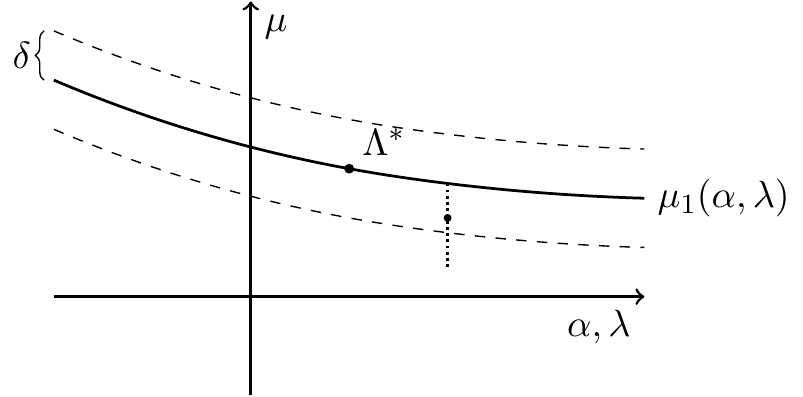}
            \caption{All nontrivial solutions near $(0,\Lambda^*)$ can be found by bifurcation from points on the graph of $\mu_1$.}
            \label{fig:local_description}
        \end{figure}
        
        \begin{theorem}[Local description]
            \label{thm:local_description_1d}
            The above family $\mathcal{S}$ of solution curves bifurcating from points $(0,\mu_1(\alpha,\lambda),\alpha,\lambda)$ for $(\alpha,\lambda)$ in a neighborhood of $(\alpha^*,\lambda^*)$ contains all nontrivial solutions of \Cref{eq:F_problem} in a neighborhood of $(0,\Lambda^*)$ in $X \times \mcU$.
        \end{theorem}
        \begin{proof}
            For each $\Lambda = (\mu_1(\alpha,\lambda),\alpha,\lambda)$ we have uniqueness in a set
            \[
                U(\Lambda) \coloneqq \brac*{ (w,\mu',\alpha,\lambda) \in X \times \mcU : \norm{w} \lor |\mu' - \mu_1(\alpha,\lambda)| < \delta(\Lambda) },
            \]
            in the sense that all nontrivial solutions in $U(\Lambda)$ are given by the solution curve obtained in \cite[Theorem 4.6]{Ehrnstroem2011}. Due to the regularity of the problem, and by possibly shrinking the neighborhood of $(\alpha^*,\lambda^*)$, we can assume that $\delta(\Lambda)$ is constant. It is then clear that the family $\mathcal{S}$ yields all the nontrivial solutions in the open neighborhood
            \[
                U \coloneqq \,\smashoperator{\bigcup_{\mu=\mu_1(\alpha,\lambda)}}\, U(\Lambda)
            \]
            of $(0,\Lambda^*)$, see \Cref{fig:local_description}.
        \end{proof}
        \begin{remark}
            We mentioned above that the same procedure can be performed using \Cref{1dBif} instead of \cite[Theorem 4.6]{Ehrnstroem2011} when the transversality condition is fulfilled. The implication is that, locally, the same solutions can be found through bifurcation with either $\mu$ or $\lambda$. It is not clear whether this is still the case for possible global solution curves.
        \end{remark}

\section{Two-dimensional bifurcation} \label{2dBifSection}
    For two-dimensional bifurcation we will use $\alpha$ as the second bifurcation parameter. We therefore need the following analogue of \Cref{IpCalcLemmalambda} for the parameter $\alpha$.
    \begin{lemma}
        \label{IpCalcLemmaalpha}
        Suppose that $\phi_n \in \ker \mcL(\Lambda)$, where $\phi_n$ is as defined in \eqref{phik}. Then, if $\tilde w_n = (\eta_{\phi_n},\phi_n)$ is the corresponding basis function of $Z$, we have
        \[
            \ip{D_{\alpha} \mcL(\Lambda) \phi_n}{\tilde{w}_n}_Y  = B \paren*{\frac{\sinh(\theta_n)}{\theta_n}}^2 + f(\theta_{n_j}),
        \]
        where
        \[
            B \coloneqq \frac{\pi}{\kappa} \brak[\bigg]{\frac{1}{\mu^2 \abs{\alpha}^2 \sin^2(\lambda)} - \frac{\cot(\lambda)}{2\abs{\alpha}^{1/2}}}
        \]
        and $f$ is defined by
        \[
            f(t) \coloneqq \frac{\pi}{\kappa} \begin{cases}\dfrac{t-\cosh(t)\sinh(t)}{2t^3} & t \neq 0,\\
                -\dfrac{1}{3} & t=0.
            \end{cases}
        \]
    \end{lemma}

    Suppose that $M(\Lambda^*) = \{n_1,n_2\}$, where $n_1 < n_2$, and that $\Lambda^*$ satisfies the transversality condition \eqref{lambdaNuisance} (which also appears for two-dimensional bifurcation). Denote the subspace of $X$ consisting of functions that have wavenumber $n\kappa$ in the horizontal variable by $X^{(n)}$. Then $\ker \mcL(\Lambda^*)|_{X^{(n_2)}} = \lspan{\{\phi_{n_2}^*\}}$. Using the local description from \Cref{LocClass1D}, we obtain the set $\mc S^{(n_2)}$ of all nontrivial solutions of \eqref{eq:F_problem} in a neighborhood of $(0,\Lambda^*)$ in $X^{(n_2)} \times \mcU$. Similarly, under the condition that $n_1 \nmid n_2$, the kernel of $\mcL(\Lambda)|_{X^{(n_1)}}$ is spanned by $\phi_{n_1}^*$, and we obtain the set $\mc S^{(n_1)}$ of all nontrivial solutions in a neighborhood of $(0,\Lambda^*)$ in $X^{(n_1)} \times \mcU$. The next result describes bimodal solutions near $(0,\Lambda^*)$, which are neither in $\mc S^{(n_1)}$ nor in $\mc S^{(n_2)}$.
    
    As mentioned before \Cref{IpCalcLemmaalpha}, we will use $\alpha$ as the second bifurcation parameter. Hence, we will look for solutions of the equation
    \begin{equation}
        \label{2dEq}
        \mcF(w,\mu^*,\alpha,\lambda) = 0
    \end{equation}
    for $w \in \mcO$ and $(\mu^*,\alpha,\lambda) \in \mcU$. Let therefore, for $j \in \{1,2\}$, the set $\mathcal S^{(n_j)}_{\mu^*}$ consist of all $(w,\alpha,\lambda)$ such that $(w,\mu^*,\alpha,\lambda) \in \mathcal S^{(n_j)}$.
    
    \begin{theorem}[Two-dimensional bifurcation]
        \label{2dBif}
        Suppose that $\Lambda^* \in \mcU$ is such that the transversality condition \eqref{lambdaNuisance} holds, and that
        \[
            \ker D_w \mcF(0,\Lambda^*) = \lspan{\{w_1^*, w_2^*\}},\quad w_j^* = \mcT(\Lambda^*) \phi_j^*,
        \]
        with $1 \leq n_1 < n_2$ and $\phi_j^* \coloneqq \phi_{n_j}^*$ as in \Cref{kernelLemma}. Furthermore, suppose that either $r(\Lambda^*) \notin \{0,1\}$ or $\theta(n_2,\alpha^*) = 0$ (in which case $r(\Lambda^*)=1$).
        \begin{enumerate}[(i)]
            \item If $n_1 \nmid n_2$, there exists a smooth family of nontrivial small-amplitude solutions
            \[
                \mcS_{\mu^*} \coloneqq \brac*{ \paren*{ \overline{w}(t_1,t_2),\overline{\alpha}(t_1,t_2), \overline{\lambda}(t_1,t_2) }: 0 < \abs{(t_1,t_2)} < \varepsilon }
            \]
            of \eqref{2dEq} in $\mcO \times (-\infty,0) \times (0,\pi)$, passing through $(0,\alpha^*,\lambda^*)$ when $(t_1,t_2) =0$, with
            \begin{equation}
                \label{wbar2dt}
                \overline{w}(t_1,t_2) = t_1 w_1^* + t_2 w_2^* + O(\abs{(t_1,t_2)}^2) \quad \text{in $X$ as $(t_1,t_2) \to 0$.}
            \end{equation}
            In a neighborhood of $(0,\alpha^*,\lambda^*)$ in $\mcO \times (-\infty,0) \times (0,\pi)$, the union $\mcS_{\mu^*} \cup \mcS_{\mu^*}^{(n_1)} \cup \mcS_{\mu^*}^{(n_2)}$ captures all nontrivial solutions of \eqref{2dEq}.
            
            \item Let $0 < \delta < 1$. If $n_1 \mid n_2$, there exists a smooth family of nontrivial small-amplitude solutions
            \[
                \mcS_{\mu^*}^\delta \coloneqq \brac*{ \paren*{ \overline{w}(r,v),\overline{\alpha}(r,v),\overline{\lambda}(r,v) }: 0 < r < \varepsilon,\, \abs{\sin(v)} >\delta}
            \]
            of \eqref{2dEq} in $\mcO \times (-\infty,0) \times (0,\pi)$, passing through $(0,\alpha^*,\lambda^*)$ when $r=0$, with
            \[
                \overline{w}(r,v) = r \cos(v) w_1^* + r \sin(v) w_2^* + O(r^2)\quad \text{in $X$ as $r \to 0$.}
            \]
            In a neighborhood of $(0,\alpha^*,\lambda^*)$ in $\mcO \times (-\infty,0) \times (0,\pi)$, the set $\mcS_{\mu^*}^\delta \cup \mcS_{\mu^*}^{(n_2)}$ contains all nontrivial solutions of \eqref{2dEq} such that $\abs{\sin(v)}>\delta$ in their projection $r \cos(v) w_1^*+r\sin(v) w_2^*$ on $\ker D_w \mcF(0,\Lambda^*)$ along $X_0$.
        \end{enumerate}
    \end{theorem}
    \begin{proof}
        The Lyapunov--Schmidt reduction in \Cref{LyaSch} tells us that there is a neighborhood of $(0,\Lambda^*)$ in $X \times \mcU$ in which \eqref{eq:F_problem} is equivalent to the equation $\Phi(t_1,t_2,\Lambda)= 0$, where
        \begin{equation}
            \label{Phi2dDef}
            \Phi(t_1,t_2,\Lambda) \coloneqq \Pi_Z \mcF(t_1 w^*_1 + t_2 w^*_2 + \psi(t_1 w^*_1+t_2 w^*_2,\Lambda),\Lambda).      
        \end{equation}
        Recall that $Z = \lspan{\{\tilde{w}_1^*,\tilde{w}_2^*\}}$, where $\tilde{w}_j \coloneqq (\eta_{\phi_j^*},\phi_j^*)$. If we let $\Pi_j$ denote the projection onto the span of $\tilde{w}_j^*$ along the image of $\mcL(\Lambda^*)$, then $\Pi_Z = \Pi_1 + \Pi_2$. Defining $\Phi_j \coloneqq \Pi_j \Phi$ for $j=1,2$, the equation $\Phi(t_1,t_2,\Lambda) = 0$ can then be rewritten as the system of equations
        \begin{equation}
            \label{eq:Phi_1_Phi_2}
            \Phi_1(t_1,t_2,\Lambda) = 0, \qquad \Phi_2(t_1,t_2,\Lambda) = 0.
        \end{equation}
        
        Let us first consider case \emph{(i)}, where $n_1 \nmid n_2$. We claim that
        \begin{equation}
            \label{Phi(i)equas}
            \Phi_1(0,t_2,\Lambda) =  0, \qquad \Phi_2(t_1,0,\Lambda) = 0,
        \end{equation}
        for all $t_1$, $t_2$ and $\Lambda$. We will show the first identity; the proof of the second being similar. Using that $\ker D_w \mcF(0,\Lambda^*)|_{X^{(n_2)}} = \lspan{\{w_2^*\}}$, an application of the Lyapunov--Schmidt reduction in $X^{(n_2)}$ yields a function $\tilde{\psi}$ mapping into $X_0^{(n_2)}$ satisfying
        \begin{equation} \label{starArgument}
            (I-\Pi_2) \mcF(t_2 w_2^*+\tilde{\psi}(t_2 w^*_2,\Lambda),\Lambda) = 0,
        \end{equation}
        for all $t_2$ in a neighborhood of $0 \in \R$. Due to the $2\pi/(n_2\kappa)$-periodicity of $\mcF(t w_2^*+\tilde{\psi}(t w^*_2,\Lambda),\Lambda)$, \eqref{starArgument} holds with $\Pi_2$ replaced by $\Pi_Z$, whence uniqueness of $\psi$ yields $\psi(t_2w_2^*) = \tilde{\psi}(t_2w_2^*)$. The first identity in \Cref{Phi(i)equas} then follows by definition of $\Phi_1$.
        
        Following the proof of \Cref{1dBif}, we now introduce the function $\Psi_1$ by
        \begin{equation}
            \label{PsiDef}
            \Psi_1(t_1,t_2,\Lambda) \coloneqq \int_0^1 \Phi_{1t_1}(zt_1,t_2,\Lambda)\dz,
        \end{equation}
        and similarly the function $\Psi_2$. Then $\Phi_j = t_j \Psi_j$ by \eqref{Phi(i)equas}, and we get different cases for \eqref{eq:Phi_1_Phi_2} depending on the values of $t_1$ and $t_2$. If both $t_1$ and $t_2$ vanish, we get trivial solutions. When $t_2 = 0$ but $t_1 \neq 0$, we know from \eqref{Phi(i)equas} that the system reduces to $\Phi_1(t_1,0,\Lambda) = 0$, the solutions of which correspond to $\mc S^{(n_1)}$, while when $t_1 = 0$ but $t_2 \neq 0$ we get solutions of \eqref{eq:F_problem} that lie in $\mc S^{(n_2)}$. The remaining case is where both $t_1$ and $t_2$ are allowed to be nonzero, which amounts to investigating the solutions of
        \[
            \Psi_1(t_1,t_2,\Lambda) = 0, \qquad \Psi_2(t_1,t_2,\Lambda) = 0
        \]
        in a neighborhood of $(0,0,\Lambda^*)$. We will look at those solutions with $\mu = \mu^*$.
        
        For this, we intend to use the implicit function theorem. Note that from \eqref{PsiDef} we get $\Psi_j(0,0,\Lambda^*) = \Pi_j \Phi_{t_j}(0,0,\Lambda^*)$, and that by definition \eqref{Phi2dDef} of $\Phi$ we have
        \[
            \Phi_{t_j}(t_1,t_2,\Lambda^*) = \Pi_Z D_w\mcF(w + \psi(w,\Lambda^*),\Lambda^*)(w_j^*+D_w\psi(w,\Lambda^*)w_j^*),
        \]
        where $w \coloneqq t_1 w^*_1 + t_2 w^*_2$. Using the properties of $\psi$ and \Cref{mcLisFred}, we can therefore conclude that $\Psi_j(0,0,\Lambda^*) = 0$. It remains to show that the derivative $D_{(\alpha,\lambda)}(\Psi_1,\Psi_2)(0,0,\Lambda^*)$ is invertible.
        
        Using analogous computations to those after \Cref{eq:Phi_t_lambda} in the proof of \Cref{1dBif}, we have
        \[
            \Psi_{j\beta}(0,0,\Lambda^*) = \Pi_j D_\beta \mcL(\Lambda^*)\phi_j =\frac{\ip{D_\beta \mcL(\Lambda^*)\phi_j}{\tilde{w}_j^*}_Y}{\norm{\tilde{w}_j^*}_Y^2}\tilde{w}_j^*
        \]
        for $j=1,2$ and $\beta = \alpha,\lambda$. It follows that $D_{(\alpha,\lambda)}(\Psi_1,\Psi_2)(0,0,\Lambda^*)$ is invertible if and only if the determinant
        \begin{equation}
            \label{eq:2d_determinant}
            C \coloneqq \begin{vmatrix}
                \ip{D_{\lambda} \mcL(\Lambda^*) \phi_1^*}{\tilde{w}_1^*}_Y & \ip{D_{\lambda} \mcL(\Lambda^*) \phi_2^*}{\tilde{w}_2^*}_Y \\[2pt]
                \ip{D_{\alpha} \mcL(\Lambda^*) \phi_1^*}{\tilde{w}_1^*}_Y & \ip{D_{\alpha} \mcL(\Lambda^*) \phi_2^*}{\tilde{w}_2^*}_Y
            \end{vmatrix}
        \end{equation}
        is nonzero. The inner products appearing in this determinant have already been computed in \Cref{IpCalcLemmalambda} and \Cref{IpCalcLemmaalpha}. Using these, and elementary properties of determinants, we have
        \[
            C = A \paren[\bigg]{ \paren*{ \frac{\sinh(\theta_{n_1})}{\theta_{n_1}} }^2 f(\theta_{n_2}) - \paren*{ \frac{\sinh(\theta_{n_2})}{\theta_{n_2}} }^2 f(\theta_{n_1}) }.
        \]
        
        Observe that of $\theta_{n_1}$ and $\theta_{n_2}$, only $\theta_{n_2}$ can vanish. Suppose for the moment that also $\theta_{n_2} \neq 0$. Then
        \begin{align*}
            f(\theta_{n_j}) &= \frac{\pi}{2\kappa\theta_{n_j}^2}\paren*{\frac{\sinh(\theta_{n_j})}{\theta_{n_j}}}^2\paren[\bigg]{\paren*{\frac{\theta_{n_j}}{\sinh(\theta_{n_j})}}^2 - \theta_{n_j}\frac{\cosh(\theta_{n_j})}{\sinh(\theta_{n_j})}}\\
            &=\frac{\pi}{2\kappa\theta_{n_j}^2}\paren*{\frac{\sinh(\theta_{n_j})}{\theta_{n_j}}}^2\paren[\big]{r(\Lambda^*)^2-\theta_{n_j}^2 -r(\Lambda^*)},
        \end{align*}
        whence the determinant in \Cref{eq:2d_determinant} can be written as
        \[
            \label{eq:2d_determinant_simplified}
            C = \frac{\pi A}{2\kappa}\paren*{\frac{\sinh(\theta_{n_1})}{\theta_{n_1}}}^2\paren*{\frac{\sinh(\theta_{n_2})}{\theta_{n_2}}}^2 r(\Lambda^*)(r(\Lambda^*)-1)\paren[\bigg]{\frac{1}{\theta_{n_2}^2}-\frac{1}{\theta_{n_1}^2}},
        \]
        where $A$ is nonzero due to the assumption of transversality. Hence, we immediately see that $C$ is nonzero if and only if $r(\Lambda^*) \notin \{0,1\}$. A similar computation shows that
        \[
            C = \frac{\pi A}{6\kappa}\paren*{\frac{\sinh(\theta_{n_1})}{\theta_{n_1}}}^2 \neq 0
        \]
        when $\theta_{n_2} = 0$. This concludes the proof of part $(i)$.
        
        Next, we move on to case \emph{(ii)}, where $n_1 \mid n_2$. We still find that
        \[
            \Phi_1(0,t_2,\Lambda) = 0
        \]
        for all $t_2$ and $\Lambda$, and so we can introduce
        \[
            \Psi_1(r,v,\Lambda) \coloneqq \int_0^1 \Phi_{1t_1}(zr\cos(v),r\sin(v),\Lambda)\dz
        \]
        as before; only now written using the polar coordinates $(t_1,t_2)=re^{iv}$ (identifying $\C$ and $\R^2$). Then $\Phi_1 = t_1\Psi_1$. For $\Phi_2$, the corresponding identity in \eqref{Phi(i)equas} is no longer true in general, but we still have $\Phi_2(0,0,\Lambda) = 0$. We therefore introduce $\Psi_2$ through 
        \[
            \Psi_2(r,v,\Lambda) \coloneqq \int_0^1 \brak*{\Phi_{2t_1}(zre^{iv},\Lambda) \cos(v) + \Phi_{2t_2}(zre^{iv},\Lambda)\sin(v)} \mathrm dz,
        \]
        which yields $\Phi_2 = r \Psi_2$.
        
        Like for case \emph{(i)}, the solutions of $\Phi(0,t_2,\Lambda) = 0$ near $(0,\Lambda^*)$ for $t_2 \neq 0$ correspond to solutions in $S^{(n_2)}$. When $t_1 \neq 0$, also $r \neq 0$, and so \eqref{eq:Phi_1_Phi_2} is equivalent to the problem
        \[
            \Psi_1(r,v,\Lambda)= 0,\qquad \Psi_2(r,v,\Lambda)=0,
        \]
        which we will now consider. Again, we will use the implicit function theorem to find solutions with $\mu=\mu^*$. Due to similar computations as those for case \emph{(i)}, we have $\Psi_1(0,v,\Lambda^*) = \Psi_2(0,v,\Lambda^*) =0$ and
        \begin{equation}
            \label{eq:Psi_1_derivatives_case_2}
            \Psi_{1\beta}(0,v,\Lambda^*) = \frac{\ip{D_{\beta} \mcL(\Lambda^*) \phi_{1}^*}{\tilde{w}_1^*}_Y}{\norm{\tilde{w}_1^*}_Y^2}\tilde{w}_1^*
        \end{equation}
        for all $v$ and $\beta = \alpha,\lambda$. To find the derivatives of $\Psi_2$, note that
        \[
            \Psi_{2\beta}(0,v,\Lambda^*) =\Phi_{2t_1\beta}(0,0,\Lambda^*) \cos(v) + \Phi_{2t_2\beta}(0,0,\Lambda^*) \sin(v),
        \]
        and so
        \begin{equation}
            \label{eq:Psi_2_derivatives_case_2}
            \begin{aligned}
                \Psi_{2\beta}(0,v,\Lambda^*) &= \frac{\ip{D_{\beta} \mcL(\Lambda^*) (\cos(v)\phi_1^* + \sin(v)\phi_2^*)}{\tilde{w}_2^*}_Y}{\norm{\tilde{w}_2^*}_Y^2}\tilde{w}_2^*\\
                &= \frac{\ip{D_{\beta} \mcL(\Lambda^*) \phi_2^*}{\tilde{w}_2^*}_Y}{\norm{\tilde{w}_2^*}_Y^2} \sin(v)\tilde{w}_2^*,
            \end{aligned}
        \end{equation}
        where we have used that $\cos(n_1 \kappa x)$ and $\cos(n_2 \kappa x)$ are orthogonal in $L_\kappa^2(\R)$.
        
        From the preceding, we see that the derivative $D_{(\alpha,\lambda)}(\Psi_1,\Psi_2)(0,v,\Lambda^*)$ is invertible if and only if the determinant
        \begin{equation}
            \label{eq:2d_determinant_case_2}
            \begin{aligned}
                \tilde{C} &\coloneqq \begin{vmatrix}
                    \ip{D_{\lambda} \mcL(\Lambda^*) \phi_1^*}{\tilde{w}_1^*}_Y & \ip{D_{\lambda} \mcL(\Lambda^*) \phi_2^*}{\tilde{w}_2^*}_Y\sin(v) \\[2pt]
                    \ip{D_{\alpha} \mcL(\Lambda^*) \phi_1^*}{\tilde{w}_1^*}_Y & \ip{D_{\alpha} \mcL(\Lambda^*) \phi_2^*}{\tilde{w}_2^*}_Y\sin(v)\\
                \end{vmatrix}\\
                &= C\sin(v)
            \end{aligned}
        \end{equation}
        is nonzero, where $C$ is the determinant introduced in \eqref{eq:2d_determinant}. We know that $C \neq 0$ under the assumptions of the theorem, so we can apply the implicit function theorem at $(0,v)$ if $\sin(v) \neq 0$. This can be done uniformly in $v$ as long as $\sin(v)$ is bounded away from zero.
    \end{proof}
    
    \begin{remark}
        In case $(i)$, the surface profiles in $\mcS_{\mu^*} \setminus \cup_j \mcS_{\mu^*}^{(n_j)}$ have multiple crests and troughs in each minimal period, at least when $(t_1,t_2)$ is sufficiently small. This follows from the asymptotic formula in \eqref{wbar2dt}.
    \end{remark}
    
    \begin{remark}
        Observe that the second special case listed in \Cref{example:explicit_kernels} has $n_1 \mid n_2$, while the third has $n_1 \nmid n_2$. They therefore fall into different cases in \Cref{2dBif}.
    \end{remark}
    
    \subsection{Properties of the bifurcation sheet}
        We will now present some properties of the sheets of solutions that were found in the two-dimensional bifurcation result, \Cref{2dBif}, following the lines of \Cref{1dTaylorSec}. The main purpose of this is to show that these sheets, found by bifurcating with respect to $\lambda$ and $\alpha$, do not, in general, coincide with the sheets found in \cite[Theorem 4.8]{Ehrnstroem2011}. Like for one-dimensional bifurcation, \Cref{2dBif} differs from the one in \cite{Ehrnstroem2011} by the use of $\lambda$ instead of $\mu$, and the addition of the transversality condition \eqref{lambdaNuisance}.
        
        The first step towards showing that the sheets differ is \Cref{gradZero}, which is the two-dimensional counterpart of \Cref{lambdaDotZero}.
        
        \begin{proposition}[Gradients of $\overline{\alpha}$ and $\overline{\lambda}$] \label{gradZero}
            For the solution sheets obtained in \Cref{2dBif}, we have the following:
            \begin{itemize}
                \item In case (i), the solutions satisfy
                \[
                    \nabla \overline{\alpha}(0,0) = \nabla \overline{\lambda}(0,0) = 0.
                \]
                \item In case (ii), we have
                \[
                    \overline{\alpha}_r(0,v) = \overline{\lambda}_r(0,v) = 0,
                \]
                as long as $n_2 \neq 2 n_1$.
            \end{itemize}
        \end{proposition}
        \begin{proof}
            The proof for case \emph{(i)} is a simpler variant of that for case \emph{(ii)}, so we focus on the latter. By definition of $\overline \alpha$ and $\overline \lambda$, we have the identity
            \[
                \Psi_j(r,v,\mu^*,\overline \alpha(r,v), \overline \lambda(r,v)) = 0
            \]
            for $j = 1,2$. Through taking derivatives with respect to $r$, this implies that
            \begin{equation}
                \label{eq:derivative_system_case_2}
                \Psi_{jr} + \Psi_{j\alpha} \overline \alpha_r(0,v) + \Psi_{j\lambda} \overline{\lambda}_r(0, v) = 0,
            \end{equation}
            where the derivatives of $\Psi_j$ are evaluated at $(0,v,\Lambda^*)$. This linear system of equations can be solved for $\overline{\alpha}_r(0,v)$ and $\overline{\lambda}_r(0,v)$ because the determinant $\tilde{C}$ in \Cref{eq:2d_determinant_case_2} is nonzero. In order to show that the derivatives vanish, it is therefore sufficient (and necessary) to show that $\Psi_{jr}(0,v,\Lambda^*)=0$ for $j=1,2$.
            
            Using the definitions of $\Psi_1$ and $\Psi_2$, we find
            \begin{align*}
                \Psi_{1r}(0,v,\Lambda^*) &= \frac{1}{2} \cos(v) \Phi_{1t_1 t_1}(0,0,\Lambda^*) + \sin(v) \Phi_{1t_1 t_2}(0,0,\Lambda^*), \\
                \Psi_{2r}(0,v,\Lambda^*) &= \begin{multlined}[t] \frac{1}{2} \cos^2(v) \Phi_{2t_1 t_1}(0,0,\Lambda^*) + \frac{1}{2} \sin^2(v) \Phi_{2t_2 t_2}(0,0,\Lambda^*) \\[3pt] + \sin(v) \cos(v) \Phi_{2t_1 t_2}(0,0,\Lambda^*),\end{multlined}
            \end{align*}    
            where
            \begin{equation}
                \label{ijlDerivative}
                \begin{aligned}
                    \Phi_{lt_it_j}(0,0,\Lambda^*) &= \Pi_l  D^2_w \mathcal{F}(0,\Lambda^*)(w_i^*, w_j^*)\\
                    &=\frac{\ip[\big]{D_w^2 \mcF(0,\Lambda^*)(w_i^* ,w_j^*)}{\tilde{w}_l^*}_Y}{\norm{\tilde{w}_l^*}_Y^2}\tilde{w}_l^*
                \end{aligned}
            \end{equation}
            for $i,j,l=1,2$. Using orthogonality in $L^2_\kappa(\R)$, like in the proof \Cref{lambdaDotZero}, one can show that the derivatives in \Cref{ijlDerivative} are zero, except possibly when $n_l=n_i + n_j$ or $n_l = |n_i-n_j|$. This is only the case when $n_2 = 2n_1$ and either $i=j=1$ and $l=2$ or $i \neq j$ and $l=1$.
        \end{proof}
        
        We now show that $\overline{\alpha}_r(0,v)$ and $\overline{\lambda}_r(0,v)$ can indeed be nonzero when $n_2 =2n_1$, which is not covered by \Cref{gradZero}, by considering the second special case listed in \Cref{example:explicit_kernels}.
        
        \begin{proposition}[Special case]
            \label{prop:special_case_ii}
            Let $\sigma$ be the smallest positive solution of $x \cot(x) = 1$. When $\kappa = \sigma/\sqrt{3}$ and $\Lambda^* = \paren*{1/(2\kappa),-4\kappa^2,\pi/2}$, we have
            \begin{align*}
                \Psi_{1r}(0,v,\Lambda^*) &= \paren*{1+\frac{1}{3}\kappa^2}\sin(v)\tilde{w}_1^*,\\
                \Psi_{2r}(0,v,\Lambda^*) &= \paren*{\frac{1}{16}+\frac{1}{2}\cos^2(\sigma)}\cos^2(v)\tilde{w}_2^*,
            \end{align*}
            and so by \eqref{eq:derivative_system_case_2}, \eqref{eq:Psi_1_derivatives_case_2} and \eqref{eq:Psi_2_derivatives_case_2} that
            \[
                \begin{bmatrix}
                    \overline{\alpha}_r(0,v)\\
                    \overline{\lambda}_r(0,v)
                \end{bmatrix}
                =
                \mathbb{M}
                \begin{bmatrix}
                    \paren*{1+\frac{1}{3}\kappa^2}\sin(v) \\
                    \paren*{\frac{1}{16}+\frac{1}{2}\cos^2(\sigma)}\cos(v)\cot(v)
                \end{bmatrix},        
            \]
            for a nonsingular matrix $\mathbb{M}$ not depending on $v$. In particular, $\overline{\alpha}_r(0,v)$ and $\overline{\lambda}_r(0,v)$ are both nonzero, except possibly for isolated values of $v$.
        \end{proposition}
        
        \Cref{prop:special_case_ii} shows that the sheets obtained in \Cref{2dBif} are, in general, not the same as those obtained in \cite{Ehrnstroem2011}---at least when $n_2 = 2 n_1$.
        
    \subsection{Local description of solutions}
    We finish by using \cite[Theorem 4.8]{Ehrnstroem2011} to prove a two-dimensional version of \Cref{thm:local_description_1d}, describing all nontrivial solutions in a neighborhood of a point falling into case \emph{(i)}. Let therefore $\Lambda^* \in \mcU$ be such that $M(\Lambda^*) = \{n_1,n_2\}$ with $n_1 < n_2$ and $n_1 \nmid n_2$, and such that either $r(\Lambda^*) \notin \{0,1\}$ or $\theta(n_2,\alpha^*) = 0$.
    
    Proceeding as in \Cref{LocClass1D}, we use \Cref{theorem:kernel_local_description} to conclude that there is a neighborhood of $\Lambda$ in which
    \[
        M(\Lambda) = \begin{cases}
        \{n_1,n_2\} & \alpha = \alpha^*, \mu = \mu_*(\lambda),\\
        \{n_i\} & \alpha \neq \alpha^*, \mu = \mu_i(\alpha,\lambda),\\
        \varnothing & \text{otherwise.}
        \end{cases}
    \]
    We may now apply \cite[Theorem 4.8]{Ehrnstroem2011} to each point on the graph of $\mu_*$ (where $r(\Lambda)=r(\Lambda^*)$ and $\alpha = \alpha^*$), obtaining a family $\mc S$ of bifurcating solution sheets. In addition, one has the solutions in $\mcS^{(n_1)}$ and $\mcS^{(n_2)}$, which were described before \Cref{2dBif}. These are all the nontrivial solutions near $(0,\Lambda^*)$. We omit the proof, which is essentially the same as for \Cref{thm:local_description_1d}.
    
    \begin{theorem}[Local description] The family $\mathcal{S}$ of solution sheets bifurcating from points $(0,\mu_*(\lambda),\alpha^*,\lambda)$ for $\lambda$ in a neighborhood of $\lambda^*$, together with the families $\mc S^{(n_1)}$ and $\mc S^{(n_2)}$, constitutes all nontrivial solutions in a neighborhood of $(0,\Lambda^*)$ in $X \times \mcU$.
    \end{theorem}
    \begin{remark}
        We could alternatively have used \Cref{2dBif} at points where the transversality condition is fulfilled. It follows that, locally, the same solutions can be found through bifurcation with either $\mu$ or $\lambda$.
    \end{remark}

\appendix
\section{Derivatives of \texorpdfstring{$\mcF$}{F}}
    \label{appendix}
    The purpose of this appendix is simply to record the derivatives of $\mcF$ with respect to $w$ at $(0,\Lambda)$, up to the third order. These are used to obtain derivatives of the bifurcation curves from \Cref{1dBif} and the bifurcation sheets from \Cref{2dBif}.
    
    We have
    \begin{align*}
        D_w \mcF_1(0,\Lambda)w &=(1-\psi_{0s}^2)\eta + \psi_{0s}\hat{\phi}_s\\
        D_w^2 \mcF_1(0,\Lambda)w^2 &=3\psi_{0s}^2\eta^2 + \psi_{0s}^2 \eta_x^2 - 4\psi_{0s}\eta \hat{\phi}_s + \hat{\phi}_s^2\\
        D_w^3 \mcF_1(0,\Lambda)w^3 &=-12 \psi_{0s}^2 \eta^3 - 6\psi_{0s}^2 \eta \eta_x^2 + 18\psi_{0s}\eta^2 \hat{\phi}_s + 6\psi_{0s} \eta_x^2 \hat{\phi}_s - 6\eta \hat{\phi}_s^2
    \end{align*}
    for $\mcF_1$, where we suppress the evaluation at $s=1$, and the derivatives
    \begin{align*}
        D_w \mcF_2(0,\Lambda)w &= - 2\psi_{0ss}\eta -s\psi_{0s} \eta_{xx} + (\partial_x^2 + \partial_s^2 - \alpha)\hat{\phi}\\
        D_w^2 \mcF_2(0,\Lambda)w^2 &=\begin{multlined}[t]6\psi_{0ss}\eta^2 + 2s\psi_{0s}\eta\eta_{xx}+(4s\psi_{0s}+2s^2\psi_{0ss})\eta_x^2 \\-4\eta\hat{\phi}_{ss} - 4s\eta_x \hat{\phi}_{xs} -2s\eta_{xx}\hat{\phi}_s\end{multlined}\\
        D_w^3 \mcF_2(0,\Lambda)w^3 &=\begin{multlined}[t]-24\psi_{0ss}\eta^3 -6s\psi_{0s}\eta^2\eta_{xx} -(24s\psi_{0s}+12s^2 \psi_{0ss})\eta\eta_x^2\\ +18\eta^2 \hat{\phi}_{ss} + 12s\eta\eta_x\hat{\phi}_{xs} + 6s\eta\eta_{xx}\hat{\phi}_s + 12s\eta_x^2 \hat{\phi}_s \\+6s^2\eta_x^2\hat{\phi}_{ss}\end{multlined}
    \end{align*}
    for $\mcF_2$.

\section*{Acknowledgments}
    The authors would like to thank the referees for several constructive comments and suggestions that helped improve the manuscript.

\end{document}